\documentclass[reqno]{amsart}
\title[Mandelbrot cascades and nonlinear smoothing
transforms]{Mandelbrot cascades on random weighted trees and nonlinear
  smoothing transforms}
  
\author[Barral and Peyri\`ere]
{Julien Barral$^{*}$ \and Jacques Peyri\`ere$^{\dag\ddag}$}
\thanks{$^*$ LAGA (UMR 7539), D\'epartement de Math\'ematiques,
  Institut Galil\'ee, Universit\'e Paris~13, Sorbonne Paris Cit\'e, 99
  avenue Jean-Baptiste Cl\'ement, 93430 Villetaneuse,
  France. \texttt{barral@math.univ-paris13.fr}}
\thanks{$^\dag$ Department of Mathematics and Systems Science, Beihang
  University, XueYuan Road No.37, HaiDian District, BeiJing, P.\ R.\
  China.}%
  \thanks{$^\ddag$ Universit\'e Paris-Sud, Math\'ematique b\^at.\ 425,
    CNRS UMR 8628, 91405 Orsay Cedex,
    France. \texttt{Jacques.Peyriere@math.u-psud.fr}}


\usepackage{amsmath,amsfonts,amssymb,mathrsfs}
\newtheorem{theorem}{Theorem}%[section]
\newtheorem{lemma}[theorem]{Lemma}%[section]
\newtheorem{proposition}[theorem]{Proposition}%[section]
\newtheorem{corollary}[theorem]{Corollary}%[section]
\theoremstyle{definition}
\theoremstyle{remark}
\newtheorem{remark}[theorem]{Remark}%[section]
\newtheorem{remarks}[theorem]{\textbf{Remarks}}%[section]
\newcommand{\abs}[1]{\lvert #1\rvert}
\newcommand{\dif}{\mathrm{d}}
\newcommand{\moment}{\mathbf{m}}
\newcommand{\T}{\mathsf{T}}
\newcommand{\smooth}{\mathsf{S}}
\newcommand{\M}{\mathsf{M}}

\newcommand{\ind}{{\mathbf{1}}}
\newcommand{\alphabet}{\mathscr{A}}
\newcommand{\e}{\mathrm{e}}
\newcommand{\mi}{\mathrm{i}}

\DeclareMathOperator{\esp}{\mathbb{E}} %expectation
\DeclareMathOperator{\Proba}{\mathbb{P}}

\DeclareMathOperator{\esssup}{\mathrm{ess\,sup}}
\DeclareMathOperator{\var}{\mathrm{Var}}
\begin{document}

\begin{abstract}
  We consider complex Mandelbrot multiplicative cascades on a random
  weigh\-ted tree. Under suitable assumptions, this yields a dynamics
  $\T$ on laws invariant by random weighted means (the so called fixed
  points of smoothing transformations) and which have a finite moment
  of order 2. Moreover, we can exhibit two main behaviors: If the
  weights are conservative, i.e., sum up to~1 almost surely, we find a
  domain for the initial law $\mu$ such that a non-standard
  (functional) central limit theorem is valid for the orbit
  $(\T^n\mu)_{n\ge 0}$ (this completes in a non trivial way our
  previous result in the case of non-negative Mandelbrot cascades on a
  regular tree). If the weights are non conservative, we find a domain
  for the initial law $\mu$ over which $(\T^n\mu)_{n\ge 0}$ converges
  to the law of a non trivial random variable whose law turns out to
  be a fixed point of a quadratic smoothing transformation, which
  naturally extends the usual notion of (linear) smoothing
  transformation; moreover, this limit law can be built as the limit
  of a non-negative martingale. Also, the dynamics can be modified to
  build fixed points of higher degree smoothing transformations.
\end{abstract}

\maketitle
\section{Introduction}\label{sec1}
This work is a continuation of our study of iteration of Mandelbrot
cascades~\cite{BPW} in which we brought to evidence central limit
theorems. The cascades then took place on a homogeneous tree. This
time we consider such cascades on a random weighted tree, which is a
kind of random scenery. This is not a mere generalization. This
setting leads to new and somewhat unexpected phenomena. In particular
it provides solutions to a non-linear equation whose unknown is a
probability distribution, as well as a non-standard central limit
theorem.

To be more specific, let $N$ be a non-negative integer valued random
variable and $a=(a_j)_{j\ge 1}$ be a sequence of non-negative random
variables such that $$\esp \sum_{j\ge1} a_j = 1.$$

We assume that $a$ and $N$ are independent.

Let ${\mathscr L}W$ stand for the law of the random
variable~$W$. Consider the following map
\begin{equation*}
  \mu \longmapsto {\mathscr L} \sum_{j\ge 1} a_j \prod_{0\le k\le N} W_k(j),
\end{equation*}
where $\mu$ is a probability measure on~${\mathbb C}$, all variables
$W_k(j)$ are distributed according to~$\mu$, independent, and
independent of $a$ and~$N$. Under suitable assumptions on $a$ and $N$,
we prove that this mapping has a unique non trivial fixed point in a certain
domain defined by inequalities on the second moment of~$\mu$. In other
terms this solves the equation 
\begin{equation}\label{NL}
\mu = {\mathscr L} \sum_{j\ge 1} a_j \prod_{0\le k\le N} W_k(j).
\end{equation}
If $N$ is the contant~0, $\mu$ is a fixed point of the well known linear smoothing transform associated with $a$. Otherwise,  equation \eqref{NL} is nonlinear, and in case $N$ is the constant~1, it reduces to the quadratic
equation~\eqref{WW}.  If $\sum a_j=1$ with probability~$1$, the so
called conservative case, this fixed point is the Dirac mass at~$1$,
whereas in the non conservative case the fixed point is non
trivial. Moreover, the conservative case gives rise to a non standard
central limit theorem.

Non trivial solutions to equation \eqref{NL}  are obtained by iterating Mandelbrot cascades in a random scenery. This setting is explained in Section~\ref{cascades}. Sections~3 to~6 deal with the case when $N$ is the constant ~1. Indeed we prefer to present calculations and ideas in this particular case. Then, in Section~7, we treat the
general case. In Section~\ref{sec8}, we  obtain a functional central limit theorem in case $N=1$.
 
\section{Mandelbrot cascades}\label{cascades}

Consider the tree $\mathscr{T}=\bigcup_{n\ge 0}\mathbb{N}_+^n$ whose
root is the only element $\epsilon$ of $\mathbb{N}_+^0$. Endowed with
concatenation, denoted by juxtaposition except in Section~\ref{sec8}, ${\mathscr T}$ is also
a monoid whose identity is~$\epsilon$. Its elements are considered as
words: If~$w=x_1x_2\cdots x_n$, we set $|w|=n$, $w_j=x_j$, and $w|_{k}
= x_1\cdots x_k$ (with $w|_0=\epsilon$). The set $\mathbb{N}_+^n$ will sometimes be denoted by $\mathscr{T}_n$. 

We are given once for all a sequence $a=(a_n)_{n\ge 1}$ of non-negative random
variables such that $\displaystyle \esp\sum_{j\ge 1} a_j = 1$ and
$\displaystyle \esp\Big(\sum_{j\ge 1} a_j\Big)^2 <\infty$. We exclude
that $a_j\in\{0,1\}$ for all $j\ge 1$ almost surely. The law of $a$ will be used to define a random scenery in which Mandelbrot construction will be done as explained in the next subsection.

We set
\begin{equation*}\label{notations}
b = \frac{\esp \left( \sum a_j\right)^2}{\esp \sum a_j^2} \text{\quad
  and \quad} q = \frac{1}{\esp \left( \sum a_j\right)^2}.
\end{equation*}
Observe that $b\ge 1$, $0< q \le 1$, and that $q=1$ if and only if
$\sum_{j\ge 1} a_j = 1$ with probability~1. As already said, we refer
to this last case as the \emph{conservative case}.

\subsection{Scalar cascades in a random scenery}

Now, if $W$ is an integrable complex random variable of expectation~1,
we consider two independent sequences $\bigl(W(w)\bigr)_{w\in
  {\mathscr T}}$ and $\bigl(a(w)\bigr)_{w\in {\mathscr T}}$ of
independent variables equidistributed with~$W$ or~$a$, and define
\begin{equation*}
Y_n = \sum_{w=j_1j_2\dots j_n\in {\mathscr T}_{n}} \prod_{k=0}^{n-1}
a_{j_{k+1}}(w|_{k})W(w|_{k+1}). 
\end{equation*}
This is a Mandelbrot martingale.  We have
\begin{equation}\label{recur0}
Y_{n+1} = \sum_{j\ge 1} a_jW(j)Y_n(j),
\end{equation}
where the variables $Y_{n-1}(j)$ are defined as $Y_{n-1}$ but starting
from $j$ as a root, and all the variables in the sum are
independent and independent of~$a$. Notice that $Y_n(j)$ has the same
distribution as $Y_n$ for all~$j\ge 1$.
From this relation we obtain the following equality.
\begin{equation*}
\esp |Y_{n+1}|^2 = \esp |W|^2 \esp |Y_{n}|^2\esp \sum_{j\ge 1} a_j^2 + \esp
\sum_{i\ne j} a_ia_j
\end{equation*}
which can be rewritten as
\begin{equation}\label{recur}
bq\esp |Y_{n+1}|^2 = \esp |W|^2 \esp |Y_{n}|^2+b-1.
\end{equation}
This means that this martingale is bounded in
$L^2$ if and only if 
\begin{equation}\label{non-deg}
\esp |W|^2  < bq.
\end{equation}
If it is so, which we assume from now on, let~$Y$ stand for the limit
of this martingale. It results from~\eqref{recur0} that~$Y$ fulfills
the following equation.
\begin{equation}\label{eqfonc}
Y = \sum_{j\ge 1} a_jW(j)Y(j),
\end{equation}
where the variables $Y(i)$ are equidistributed with~$Y$, and all the
variables~$W(j)$ and~$Y(j)$ in the sum are independent and independent
of~$a$. Thus $Y$ is a special fixed point of the so-called smoothing
transformation $\smooth_1$ which to a given probability distribution
$\mu$ on $\mathbb C$ associates
$$
\smooth_1(\mu) = {\mathscr L}\sum_{j\ge 1} a_jW(j)Z(j),
$$
where the $Z_j$ are independent, distributed according to $\mu$, and
independent of $a$ and the $W(j)$ (${\mathscr L}X$ stands for the
probability distribution of $X$). The fixed points of such
transformations have been studied for a long time, especially in the
context of model for turbulence and branching processes, and much is
known about their structure
\cite{mandel2,KP,Biggins1,DuLi,gu90,BiKy,liu98,Liu,BiKy3,AlBiMe,AlMe12,AlMe}.

Also, we have
\begin{equation}\label{mmnt4}
\esp |Y|^2 = \frac{b-1}{bq-\esp |W|^2}.
\end{equation}

Recall that, when $W$ is a non-negative variable, for $p>1$ we have
\begin{equation}\label{liu}
\esp Y^p<\infty\Longleftrightarrow 
\begin{cases}
\esp\Big (\sum_{j\ge 1} a_jW(j)\Big )^p < \infty,\\
\esp\Big (\sum_{j\ge 1}  a_j^p W(j)^p\Big )<1.
\end{cases}
\end{equation}
This equivalence follows from Liu's generalisation~\cite{Liu} of the
condition found in~\cite{KP} for the finiteness of moments of order
larger than 1 of non-degenerate Mandelbrot cascades on regular trees.
\medskip

We denote by~$\T$ the operator which leads from the distribution of
$W$ to the one of $Y$. From time to time we shall make the following
abuse of notation: $Y=\T W$.

We wish to iterate $\T$ and study its dynamics. In \cite{BPW}, we
considered the case where $b\ge 3$ is a fixed integer, $a_j=b^{-1}$ if
$1\le j\le b$ and 0 otherwise, and $W$ is non-negative. We proved that
if one starts with $W_0$ such that $1\le \esp W_0^2<b-1$, one can
infinitely iterate $\T$, and $\T^n W_0$ converges in law to
$\delta_1$, the unit Dirac mass at~1. Then, under additional
assumptions on $W_0$, we proved that after centering $\T^n W_0$ and
normalizing it by the resulting standard deviations one gets a
sequence of probability distributions converging to the standard
normal law; moreover, this result had a functional counterpart in
which the limit process was obtained as the limit of an ``additive''
cascade.  As we will see, in the present extended framework, the
situation exhibits new features. First, when $q=1$, i.e.  $\sum_{j\ge
  1} a_j = 1$ with probability~1, there is a more general non standard
central limit theorem: the limit distribution is that of a complex
centered Gaussian variable $\xi$ multiplied by $\sqrt{U}$, where $U$
is independent of~$\xi$ and is the limit of a non degenerate
Mandelbrot martingale built on $\bigcup_{n\ge 1}
\{0,1\}^n\times\mathbb N_+^n$ rather than on $\mathbb N_+^n$, which is
an unexpected fact (Theorems~\ref{thm13} and~\ref{thm16}). This result
has a functional counterpart too (Theorem~\ref{CLTFunc}), the limit
process being the limit of a mixture between additive and
multiplicative cascade. Also, when $q<1$, we find conditions under
which there exists a non trivial fixed point of $\T$ (in the sense
that it differs from $\delta_1$) with a non trivial basin of
attraction. It turns out that this fixed point is by construction a
fixed point of the following quadratic smoothing transformation
$\smooth_2$: for a probability measure $\mu$ on~$\mathbb C$,
\begin{equation}\label{smooth2}
\smooth_2(\mu) = {\mathscr L}\left(\sum_{j\ge 1} a_j W(j)\widetilde
  W(j)\right ),
\end{equation}
where the random variables $\{W(j),\widetilde W(j)\}_{j\ge 1}$ are
mutually independent, distributed according to $\mu$, and independent
of $a$. We will identify this fixed point as the probability
distribution of the limit of a non-negative martingale
(Theorems~\ref{firstThm} and \ref{thm-3}). 

Next subsection introduces some useful preliminary facts about the
mapping $\T$.

\subsection{Simultaneous cascades}\label{simult}

%We suppose that $\varphi$ has two fixed points such that~$\beta> 1$.
This time we are given a random vector $(W,W')$ such that $\esp W =
\esp W' = 1$, $\esp |W|^2 < bq$, and $\esp |W'|^2 < bq$. We consider a
family $\bigl\{\bigl(W(w),W'(w)\bigr)\bigr\}_{w\in \mathscr{T}}$ of
independent copies of $(W,W')$, which are independent of all the
$a(w)$, and perform the same construction as previously: one gets
variables $Y_n$ and $Y_n'$ and their limits~$Y$ and~$Y'$.

Thus $\T$ extends naturally to an operation $\T^{(2)}$ mapping the
distribution of $(W,W')$ to that of $(Y,Y')$.  
\medskip

Let us perform a few computations.
\medskip

Due to~\eqref{eqfonc}
\begin{eqnarray*}
\esp \overline{Y}Y' &=&
\sum_{i,j\ge 1} \esp \Bigl(a_ia_j\overline{W(i)}W'(j)
\overline{Y(i)}Y'(j)\Bigr)\nonumber\\ 
&=& \esp (\overline{W}W')\esp (\overline{Y}Y')\esp \sum_{j\ge 1}a_j^2 +
\sum_{i\ne j} a_ia_j
\end{eqnarray*}
hence
\begin{equation}\label{mmnt3}
\esp \overline{Y}Y' = \frac{b-1}{bq - \esp \overline{W}W'}.
\end{equation}
\medskip

Again, due to~\eqref{eqfonc}

\begin{equation*}
Y-Y' = \sum_{i\ge 1} a_i\bigl(
W(i)-W'(i)\bigr)Y(i) + \sum_{i=1} a_i\bigl(
Y(i)-Y'(i)\bigr)W'(i),\\
\end{equation*}
so
\begin{multline*}
bq \esp |Y-Y'|^2 = \esp |W-W'|^2 \esp |Y|^2 + \esp |Y-Y'|^2 \esp
|W'|^2\\
+ 2\Re \esp\bigl(|(\overline{W}-\overline{W'})W'| \bigr)\esp\bigl(
|\overline{Y} (Y-Y')|\bigr).
\end{multline*}
and
\begin{equation}\label{continu}
\sqrt{bq}\ \|Y-Y'\|_2 \le \|W-W'\|_2\|Y\|_2+\|Y-Y'\|_2\|W'\|_2.
\end{equation}

\subsection{Examples}\label{examples}

The original Mandelbrot cascades correspond to the following choice
of~$a$:
\begin{equation*}
a_j = 
\begin{cases}
b^{-1},& \text{\ if\ }1\le j\le b,\\0,& \text{\ if\ } j>b.
\end{cases}
\end{equation*}

One also can associate $a$ with a Galton-Watson process. More
precisely, let $J$ be an integer valued random variable, and define
$q = \Proba \{J>0\}$ and 
\begin{equation*}
a_j=
\begin{cases}
(qJ)^{-1},& \text{\ if\ }1\le j\le J,\\0,& \text{\ if\ } j>J.
\end{cases}
\end{equation*}
In this context, it is not difficult to see that
\begin{equation*}
 b^{-1} = \esp(J^{-1} \mid
J\ne 0).
\end{equation*}
We also use the following notation: 
\begin{equation}\label{GW}
b_k = \frac1{\esp\bigl(J^{-k} \mid J\ne0\bigr)}.
\end{equation}

Notice that
\begin{equation}\label{bs}
b_1\le b_k\le b_1^k,
\end{equation}
due to H\"older inequality.

\section{A dynamical system}\label{DYN}

We wish to iterate~$\T$. So, we have to ensure that $\esp |Y|^2 <
bq$. In view of~\eqref{mmnt4} this leads first to consider the
iterates of the homography
\begin{equation}\label{homography}\varphi(x)=\frac{b-1}{bq-x}.\end{equation}

\subsection{Study of $\varphi$}

There are two cases.
\begin{enumerate}
\item The mapping $\varphi$ has no fixed point. Then one of the
  iterates of any starting point~$x_0<bq$ is larger than $bq$.

\item The mapping $\varphi$ has two real fixed points $\alpha\le
  \beta$.
  % We shall see below that $1\le \alpha\le \beta\le bq-1$.
  Starting from $x_0< \alpha$ the sequence of its iterates increases
  towards $\alpha$. Starting from $x_0\in (\alpha, \beta)$ the
  sequence decreases towards $\alpha$. Starting from $x_0> \beta$
  leads, after some iterations, to values larger than $bq$.
\end{enumerate}

This means that in case~(1) there is no hope to indefinitely
iterate~$\T$.  \medskip

As we wish to start from $x_0= \esp |W_0|^2>1$, the only interesting
case is~$\beta\ge 1$.

A first way, in case~$\beta\ge 1$, to see that $\alpha\ge 1$ is to
proceed as follows: starting with~$W=W_0$ such that $\esp |W_0|^2<
\beta$ (the case $\beta=1$ presents no interest), we get~$W_1=Y$ such
that~$\esp |W_1|^2< \beta$. So, by starting with $W_1$ instead
of~$W_0$ we get~$W_2$ which still fulfills the non-degeneracy
condition. And so on \dots As previously said, $\lim_{n\to \infty}\esp
|W_n|^2= \alpha$. But, as $\esp W_n=1$, one has $\alpha\ge 1$.

This homography, $\varphi$, has two real fixed points, $\alpha\le\beta$,
the roots of the polynomial $p(x) = x^2-bqx+b-1$, when $\displaystyle
q \ge \frac{2\sqrt{b-1}}{b}$.

Since $p(0) = p(bq) > 0$, $p(1) = b(1-q) \ge 0$, and $\alpha+\beta =
bq$, either $0<\alpha\le\beta\le 1$ or $1\le \alpha\le \beta\le
bq-1$. The first case is of no interest to us because we should start
iterating from $x_0>1$. In the first case, we have $\displaystyle
2\sqrt{b-1} \le bq\le 2$, which means~$b\le 2$. In the second case, we
have $b\ge bq = \alpha+\beta\ge 2$.  \medskip

Let us now examine the behavior of~$\varphi$ under iteration when there
are real fixed points. Suppose first $\alpha< \beta$. By using the
conservation of the cross-ratio we get
\begin{equation}\label{birap0}
  \frac{\varphi(x)-\alpha}{\varphi(x)-\beta} =
  \bigl(\varphi(x),\infty,\alpha,\beta\bigr) =
  \bigl(x,bq,\alpha,\beta\bigr) =
  \frac{\alpha}{\beta}\,\frac{x-\alpha}{x-\beta},
\end{equation}
which implies that, if $x_0<\beta$ and $x_{n+1}=\varphi(x_n)$, one has
\begin{equation}\label{birap}
x_n-\alpha=\frac{(\alpha/\beta)^n}{1-(\alpha/\beta)^n}\,
\frac{\beta-\alpha}{\beta-x_0}\,(x_0-\alpha).
\end{equation}
\bigskip

Suppose now that~$\alpha=\beta$, which means $\displaystyle b =
\frac{2(1+\sqrt{1-q^2})}{q^2}$. Then
\begin{equation*}
  \frac{\alpha-\varphi(x)}{\alpha}=\bigl( \varphi(x),0,\alpha,\infty\bigr) = 
  \bigl( x,\infty,\alpha,bq\bigr) = \frac{\alpha-x}{2\alpha-x},
\end{equation*}
gives
\begin{equation*}
\frac{1}{\alpha-\varphi(x)} = \frac{1}{\alpha}+\frac{1}{\alpha-x}.
\end{equation*}
It follows that if $x_0 <\alpha$,
\begin{equation}\label{amoinsxn}
\alpha-x_n =
\frac{\alpha(\alpha-x_0)}{n(\alpha-x_0)+1} .
\end{equation}
\medskip

The case $b=2$ is of no interest. So, \emph{\bf from now on, we
  suppose $b>2$}.  \medskip

Observe that~$\alpha\ge 1$ and that $\alpha=1$ if and only if~$q=1$.

\subsection{A dynamical system}

Let $\mathcal{P}$ be the set of Borel probability measures on
$\mathbb{C}$, and $\mathcal P^{(2)}$ the set of Borel probability
measures on $\mathbb{C}^2$.

If $\mu\in\mathcal{P}$ and $p> 0$, we denote by
$\moment_p(\mu)$ the moment of order $p$ of $\mu$, i.e.,
$$
\moment_p(\mu)=\int_{{\mathbb C}}|x|^p \,\mu(\dif x).
$$
Then let $\mathcal{P}_1$ be the set of elements of~$\mathcal{P}$ with
finite first moment and expectation~1:
$$
\mathcal{P}_1=\{\mu\in\mathcal{P}\ :\ \moment_1(\mu)< \infty,\
\int_{{\mathbb C}}
z\,\mu(\dif z) =1\}.
$$

For~$\gamma\ge1$ we set
$$ {\mathscr P}_\gamma = \bigl\{\mu\in{\mathcal P}_1\ :\ 1\le
  \moment_2(\mu)\le \gamma \bigr\}.
$$
%This is a convex compact subset of~${\mathcal P}$. 

We also set
$${\mathscr P}^{(2)}_\gamma = \bigl\{\rho\in{\mathcal P}^{(2)}\ :
\rho\circ\pi_1^{-1},\rho\circ\pi_2^{-1}\in {\mathscr P}_\gamma \bigr\}.
$$

The set ${\mathscr P}_\gamma $ is endowed with
the Wasserstein distance (see~\cite{Villani}, p.~77 sqq)
$$
\dif_{W,2}(\mu,\mu')^2 = \inf\left\{\int_{{\mathbb C}^2}|x-y|^2\,
  \dif\rho: \rho\in \mathcal P^{(2)},\ \rho\circ\pi_1^{-1}= \mu,\
  \rho\circ\pi_2^{-1}=\mu'\right\},
$$
where $\pi_1$ and $\pi_2$ stand for the canonical projections on the
first and second coordinates. The space
$\bigl(\mathcal{P}_\gamma,\dif_{W,2}(\mu,\mu')\bigr)$ is complete, and
convergence in $(\mathcal{P},\dif_{W,2}(\mu,\mu'))$ implies
convergence in distribution.
 
When $\beta>\alpha\ge 1$, for any $\gamma\in [\alpha,\beta]$, the
set~${\mathscr P}_\gamma$ is stable under operation~$\T$. This means
that we can indefinitely iterate the process on~${\mathscr
  P}_\gamma$. Similarly, ${\mathscr P}^{(2)}_\gamma$ is stable under
operation~$\T^{(2)}$ defined in Section~\ref{simult}.
  
If $\mu\in {\mathscr P}_\gamma$, due to (\ref{eqfonc}), we can
associate with each $n\ge 0$ a random variable $W_{n+1}$ as well as a
copie of $(a_j)_{j\ge 1}$ and two sequences of random variables
$(W_n(k))_{k\ge 1}$ and $(W_{n+1}(k))_{k\ge 1}$, such that the random
variables $a$, $W_n(1),W_{n+1}(1),W_n(2),W_{n+1}(2),\ldots$ are
independent,
\begin{equation}\label{foncn}
W_{n+1} = \sum_{j \ge 1} a_jW_{n}(j)W_{n+1}(j),
\end{equation}
$\T^n\!\mu $ is the probability distribution of $W_n$ and $W_n(k)$ for
every~$k\ge 1$, and $\T^{n+1}\!\mu $ is the probability distribution
of $W_{n+1}$ and $W_{n+1}(k)$ for every $k\ge 1$. One has to be aware
that, if we also write Equation~\eqref{foncn} for $W_{n+2}$, the
variables $W_{n+1}(j)$ which appear in both formula need not be the
same.

More generally, if $\rho\in {\mathscr P}^{(2)}_\gamma$, we can
associate with each $n\ge 0$ a random vector $(W_{n+1},W_{n+1}')$ as
well as a copy of $N$ and two sequences of random vectors
$\bigl((W_n(k),W'_n(k))\bigr)_{k\ge 1}$ and
$\bigl((W_{n+1}(k),W'_{n+1}(k))\bigr)_{k\ge 1}$, such that the random
vectors $a$, $(W_n(1),\,W'_n(1)),\,(W_{n+1}(1),\,W'_{n+1}(1)),\,
(W_n(2),W'_n(2)),\, (W_{n+1}(2),\,W'_{n+1}(2)),\ldots$ are
independent, and
\begin{equation}\label{foncn-1}
\begin{cases}\displaystyle W_{n+1}=
\sum_{j \ge 1}a_jW_{n}(j)W_{n+1}(j),\\ \displaystyle
W'_{n+1} = \sum_{j \ge 1}a_jW'_{n}(j)W'_{n+1}(j),
\end{cases}
\end{equation}
where $(\T^{(2)})^n\!\rho $ is the probability distribution of $(W_n,W_n')$
and $(W_n(k),W'_n(k))$ for every~$k\ge 1$, and $\T^{(2),n+1}\!\rho$ is
the probability distribution of $(W_{n+1},W'_{n+1})$ and
$(W_{n+1}(k),W'_{n+1}(k))$ for every $k\ge 1$.

\subsection{Existence of fixed points for $\T$ (or
  $\smooth_2$)}\label{existence}
It turns out that, exactly like in the case of the classical linear
smoothing transformation, it is possible to build special fixed points
of $\smooth_2$, and hence of $\T$, as limit of a martingale whose
successive terms are distributed according to $\smooth_2^n(\delta_1)$.

Let $\{a(w,m)\}_{\substack{n\ge 0\\(w,m)\in \{0,1\}^n\times \mathscr
    {T}_n}}$ be a sequence of independent copies of $a$. For all $(w,m)$ we
define $Z_1(w,m)=\sum_{j\ge 1}a_j(w,m)$, which is distributed
according to $\smooth_2(\delta_1)$. Then we define recursively, for
all $n\ge 2$ and $(w,m)$,
\begin{equation}\label{znmw}
Z_n(w,m)=\sum_{j\ge 1} a_j(w,m) Z_{n-1}(wj,m0)Z_{n-1}(wj,m1),
\end{equation}
which, as easily seen by induction, is distributed according to
$\smooth_2^n(\delta_1)$.

Set $Z_n=Z_n(\epsilon,\epsilon)$. The sequence $(Z_n)_{n\ge 1}$ is a
non-negative martingale with respect to the filtration\ \ $\mathcal
G_n=\sigma\Bigl(a(w,m): (w,m)\in \bigcup_{k=0}^{n-1}\{0,1\}^k\times
\mathscr {T}_k\Bigr)$, $n\ge 1$. To~see this, define\ \  ${\mathcal
  G}_n(w,m)= \sigma\Bigl(a(ww',mm'): (w',m')\in
\bigcup_{k=0}^{n-1}\{0,1\}^k\times \mathscr {T}_k\Bigr)$. We have
\begin{align*}
  &\esp(Z_2(w,m) \mid {\mathcal G}_1(w,m))\\&=\sum_{j_1\ge 1}a_{j_1}(w,m)
  \esp (Z_{1}(wj_1,m0) \mid {\mathcal G}_1(w,m))\esp(Z_{1}(wj_1,m1) \mid \mathcal
  G_1(w,m))\\&= \sum_{j_1\ge 1}a_{j_1}(w,m) \esp\left (\sum_{j_2\ge 1}
    a_{j_2}(wj_1,m0)\right) \esp\left (\sum_{j_2\ge 1}
    a_{j_2}(wj_1,m1)\right)\\&= \sum_{j_1\ge 1}a_{j_1}(w,m)=Z_1(w,m),
\end{align*}
Then, suppose that for a given $n\ge 3$, for all $(w,m)$ we have 
$$
\esp
(Z_{n-1}(w,m) \mid {\mathcal  G}_{n-2}(w,m))=Z_{n-2}(w,m).
$$ 
Using \eqref{znmw}
and the independence between random variables, we get
\begin{align*}
  \esp(Z_n(mw) \mid {\mathcal G}_{n-1}(w,m))&=\sum_{j\ge 1} a_j(w,m)
  \prod_{\epsilon\in\{0,1\}}\esp(Z_{n-1}(wj,m\epsilon) \mid \mathcal
  G_{n-1}(w,m))\\ 
  &=\sum_{j\ge 1} a_j(w,m)
  \prod_{\epsilon\in\{0,1\}}\esp(Z_{n-1}(wj,m\epsilon)|\mathcal
  G_{n-2}(wj,m\epsilon))\\
  &=\sum_{j\ge 1} a_j(w,m) Z_{n-2}(wj,m0)Z_{n-2}(wj,m1)=Z_{n-1}(w,m).
\end{align*}

Equation  \eqref{znmw} also yields 
$$
\esp Z_n^2 = \left (\esp\sum_{j\ge 1} a_j^2\right) \bigl(\esp
Z_{n-1}^2\bigr)^2 + \esp\sum_{i\neq j} a_ia_j=\frac{1}{bq}\bigl(\esp
Z_{n-1}^2\bigr)^2 + \frac{b-1}{bq}
$$
for all $n\ge 1$. Notice that the mapping $x \mapsto
\frac{1}{bq}x^2+\frac{b-1}{bq}$ has exactly the same fixed points as
$\varphi$, namely $\alpha$ and $\beta$. Since $\esp Z_0^2=1$ and $Z_n$
satisfies the recursion~\eqref{znmw}, we get the following result.

\begin{theorem}\label{firstThm}
  Suppose that $\alpha\ge 1$. The martingale $(Z_n)_{n\ge 1}$ is
  bounded in $L^2$, and converges to a limit $Z$ such that
  $\esp(Z^2)=\alpha$. Hence the probability distribution of~$Z$,
  denoted by~$\M$, is a fixed point of~$\smooth_2$ and~$\T$, and
  $\moment_2(M)=\alpha$.
\end{theorem}

\subsection{Basin of attraction of $\mathbf\M$, convergence speed, and
  explosion of moments when $q<1$}

\begin{lemma}\label{sim} 
\begin{enumerate}
\item Suppose $\beta> \alpha\ge 1$, fix $\gamma\in[\alpha,\beta)$
  and $\rho\in{\mathscr P}_\gamma^{(2)}$. Let $(W_n,W'_n)$ be a
  sequence of variables distributed according to $(\T^{(2)})^n(\rho)$. Then
  $\esp |W_n-W'_n|^2 = \mathrm{O}\bigl( (\alpha/\beta)^n\bigr).$
\item If $\alpha=\beta$, fix $\rho\in{\mathscr P}_\alpha^{(2)}$. Let
  $(W_n,W'_n)$ be a sequence of variables distributed according to
  $(\T^{(2)})^n(\rho)$. Then $\esp |W_n-W'_n|^2 = \mathrm{O}(1/n).$
\end{enumerate}
\end{lemma}

\proof If $\alpha < \beta$, Equation~\eqref{birap} tells that $\esp
|W_n|^2 = \alpha+ \mathrm{O}\bigl( (\alpha/\beta)^n\bigr),\ \esp
|W_n'|^2 = \alpha+ \mathrm{O}\bigl( (\alpha/\beta)^n\bigr)$, and $\esp
\overline{W}_nW_n' = \alpha+ \mathrm{O}\bigl( (\alpha/\beta)^n\bigr).$
So, $$ \esp |W_n-W_n'|^2 = \esp |W_n|^2 + \esp |W_n'|^2 - 2\Re \esp
\overline{W}_nW_n' = \mathrm{O}\left( (\alpha/\beta)^n\right).$$

For the second assertion, use Equation~\eqref{amoinsxn} instead
of~\eqref{birap}.

\begin{theorem}\label{thm-3}
  Suppose that $\beta \ge \alpha\ge 1$. 
\begin{enumerate}
\item The probability~$M$ of Theorem~\ref{firstThm} is the unique
  fixed point of~$\T$ in ${\mathscr P}_\alpha$.
\item If $\alpha=\beta$, then, for all~$\mu\in {\mathscr P}_\alpha$,
  $\dif_{W,2}(\T^n\mu,M)=\mathrm{O}(1/n)$.
\item If $\beta>\alpha$, then, for all $\mu\in \bigcup_{\alpha\le
    \gamma<\beta}{\mathscr P}_\gamma$, $\dif_{W,2}(\T^n\mu,M)=
  \mathrm{O}\bigl((\alpha/\beta)^n\bigr)$.
\item The fixed point~$M$ is a solution to the following equation:
\begin{equation}\label{WW}
W = \sum_{j\ge 1}a_jW(j)\widetilde W(j).
\end{equation}
where $W$ is distributed according to $M$, the $W(j),\widetilde W(j)$ 
are independent copies of $W$, also independent of $a$.
\end{enumerate}
\end{theorem}

\proof Fix $\gamma\in [\alpha,\beta)$ if $\alpha>\beta$ or
$\gamma=\alpha$ if $\alpha=\beta$.

Take $\mu\in {\mathscr P}_\gamma$. Let $\rho\in {\mathscr
  P}_\gamma^{(2)}$ such that $\pi_1(\rho)=\mu$ and $\pi_2(\rho)=M$
(where $\M$ is defined in Theorem~\ref{firstThm}). Due to
Lemma~\ref{sim} and the fact that $TM=M$, we get that the Wasserstein
distance between $T^n\mu$ and $M$ tends to 0, with the speed claimed
in the statements.
  
When $\alpha<\beta$, one can give an alternate proof of the existence
of a fixed point. Indeed, take $\mu\in {\mathscr P}_\gamma$ and set
$\mu'=\T\mu$. It follows from Lemma~\ref{sim} that $\esp
|W_{n+1}-W'_{n+1}|^2$ converges exponentially to $0$. Consequently, so
does the Wasserstein distance between $\T^{n+1}\mu$ and
$\T^{n+1}\mu'=\T^{n+2}\mu$. It follows that $(\T^n\mu )_{n\ge 0}$ is a
Cauchy sequence in $\mathscr P_\gamma$ endowed with $\dif_{W,2}$, so
$\T^n\mu$ converges in distribution as $n\to\infty$, to a limit law
$\M(\mu)$, obviously in ${\mathscr
  P}_\alpha$. Equation~\eqref{continu} implies that $\T$ is
continuous, so $\M(\mu)$ is a fixed point. Lemma~\ref{sim}
yields uniqueness.

When~$q=1$, we have $\alpha=1$. So, in this case, the fixed point is
the Dirac mass at~1.
\medskip

The next theorem deals with the explosion of moments of~$\M$.
\begin{theorem}
  Suppose $\alpha>1$ (which means $q<1$). Then there exists $2\le p_0<
  \infty$ such that, if $W$ is distributed according to~$\M$,
$$\esp W^p< \infty \Longleftrightarrow p\le p_0.$$
\end{theorem}
\proof

Due to~\eqref{liu} and the above observation, we have
$$
\esp W^p<\infty\Longleftrightarrow 
\begin{cases}
\esp\Big( \sum_{j\ge 1} a_jW(j)\Big )^p<\infty,\\
\esp\sum_{j\ge 1} a_j^{p}W(j)^p<1.
\end{cases}
$$
But
$$
\begin{cases}
  \esp \bigl(\sum_{j\ge 1}a_j\bigr)^p\le \esp\Big(\sum_{j\ge
    1} a_jW(j)\Big )^p\\
\sum_{j\ge 1}^N a_j^p W(j)^p = (\esp W^p) \esp\sum_{j\ge 1} a_j^p,
\end{cases}
$$
where we used conditional expectation with respect to $\sigma(a_j:j\ge
1)$ and Jensen's inequality to get the first inequality. So
$$ \esp W^p<\infty\Longleftrightarrow
\begin{cases} \esp \bigl(\sum_{j\ge 1}a_j\bigr)^p<\infty\\
 \esp W^p < \left( \esp\sum_{j\ge 1}
a_j^p\right)^{-1}
\end{cases}
.$$
Suppose all the moments of $W$ are finite, we must have $\bigl(\esp
W^p\bigr)^{1/p} < \left( \esp\sum_{j\ge 1} a_j^p\right)^{-1/p}$ for
all $p$.  This imposes $\esssup W\le \bigl(\sup_{j\ge 1}\esssup
a_j\bigr)^{-1}$. Let $m= \esssup W$. Let~$\varepsilon>0$. By
using~\eqref{WW} we see that, for all~$n$, with positive probability,
we have $W \ge (m-\varepsilon)^2\sum_{j=1}^n a_j$. This means $m\ge
(m-\epsilon)^2 \esssup \sum_{j>1} a_j$, hence $m\le 1/\esssup
\sum_{j>1} a_j$. But as $\sum_{j\ge 1} a_j$ is not constant and of
expectation~$1$, we would have $m<1$, which is impossible, since~$\esp
W = 1$.  \medskip

Define $p_0=\sup\{p\ge 2: \esp W^p<\infty\}$, and $\Phi(p) = (\esp
W^p) \esp\sum_{j\ge 1} a_j^p$. We necessarily have $\Phi(p_0)<1$, if
$\esp(W^{p_0})<\infty$ or $\Phi(p_0)=\infty$ if
$\esp(W^{p_0})=\infty$. However, $\Phi$ is lower semi-continuous, so
$\Phi(p_0)=\infty$ is impossible, for otherwise $\Phi(p)$ should tend
to $\infty$ as $p$ tends to $p_0$ from below, while it is bounded
by~1.  \medskip

Proposition~\ref{no3rdmoment}, in the next section, gives examples for
which~$p_0<3$.

\section{Moments of order 3}

In this section we suppose that $\beta>1$ and that $W_0$ is a
\emph{non-negative} random variable.

Define $u$, $v$ and $w$ as follows, and suppose these quantities are
finite:
\begin{eqnarray}\label{moments-a}
\frac{1}{u} &=& \esp\sum a_j^3\nonumber\\
\frac{1}{v} &=& \esp\sum_{i\ne j} a_i^2a_j\\
\frac{1}{w} &=& \esp \sum_{\#\{i,j,k\}=3} a_ia_ja_k.\nonumber
\end{eqnarray}

Notice that we have $\displaystyle \frac1u+\frac3v+\frac1w= \esp
\left( \sum a_j\right)^3$, which in the conservative case implies
$\displaystyle \frac1u+\frac3v+\frac1w= 1$. 

\medskip
Also, H\"older inequality yields
\begin{equation*}
\left( \esp \sum a_j^2\right)^{1/2}\le \left( \esp \sum
  a_j\right)^{1/4}\left( \esp \sum a_j^3\right)^{1/4}.
\end{equation*}

We also set, for $\kappa>0$,
\begin{equation}\label{us}
u_\kappa =\left(\esp  \sum_{j\ge 1} a_j^\kappa\right)^{-1}.
\end{equation}
So we have $u=u_3$.

In the conservative case
\begin{equation}\label{bu}
b=u_2\text{\quad and\quad } b\le u\le b^2.
\end{equation}
\medskip

It is easy to get the following formula from Equation~\eqref{eqfonc}
\begin{equation}\label{recur3}
\esp(Y^3) = \esp W^3\esp Y^3/u+3\esp W^2\esp Y^2/v + 1/w
\end{equation}
which can be written as
\begin{equation}\label{Recur3}
  \esp Y^3 = \frac{u(3w\esp W^2\esp Y^2+v)}{vw(u-\esp W^3)}.
\end{equation}
Set
\begin{equation}\label{psi}
  \psi_\theta(t) = \frac{u(3w\theta\varphi(\theta)+v)}{vw(u-t)}.
\end{equation}
where $1<\theta<\beta$, and
\begin{equation}\label{map}
  \Phi\,:\ (\theta,t) \longmapsto \bigl( \varphi(\theta),\psi_\theta(t)\bigr).
\end{equation}

This means that, if $\theta=\esp W^2$ and~$t=\esp W^3$, we have
\begin{equation*}
\left( \esp Y^2,\esp Y^3\right) = \Phi(\theta,t).
\end{equation*}

This is why we wish to iterate~$\Phi$.
\medskip

\begin{remarks}\label{obvious}
Let us first make some simple observations on homographies. Consider
$\chi(x) = c/(d-x)$, where $c$ and $d$ are positive parameters. Then
\begin{enumerate}
\item $x\chi(x)$ increases with $x$ for $x\in (-\infty,d)$,
\item when $d^2>4c$, $\chi$ has two real fixed points; when~$d$ is
  fixed, the smaller fixed point $w_-$ is an increasing function
  of~$c$ and the larger one $w_+$ is decreasing.
\item If $x<w_-$ then $x<\chi(x)$, and if $w_-<x<w_+$ then
  $w_-<\chi(x)<x$.
\end{enumerate}
\end{remarks}

First, one can check %(do not forget that $\theta<\beta<bq$ ,and $b_2>b>2$)
that, when $v(uw-4)+12w(b-1)>0$, $\psi_\theta$ has real fixed points if and
only if
\begin{equation*}
  \theta  \le \frac{vbq(uw-4)}{v(uw-4)+12w(b-1)}.
\end{equation*}
If it is so, let $\gamma_{{}_-}(\theta)\le \gamma_{{}_+}(\theta)$
stand for the fixed points.\medskip

Define
\begin{equation}\label{theta}
  \theta_a = \frac{vbq(uw-4)}{v(uw-4)+12w(b-1)}
  \text{\quad and\quad}   \vartheta = \min \{\beta, \theta_a\}.
\end{equation}

In the Galton-Watson case, $v(uw-4)+12w(b-1)$ has the same sign as
$b_2^2q^4+8b_2-12b_1+4$. But, since $2<b_1<b_2$ (see~\eqref{bs}) and
$q\le 1$, we have 
\begin{equation*}
b_2^2+8b_2-12b_1+4 > b_2^2-4b_2+4>0.
\end{equation*}
So, if $q$ is large enough $b_2^2q^4+8b_2-12b_1+4$ is positive.

In this setting we have
$$  \theta_a = \frac{b_1b_2^2q^5+4(3-b_1)b_2q-8b_1q}{b_2^2q^4+8b_2-12b_1+4}.$$

\begin{proposition}\label{no3rdmoment}
  If 
\begin{enumerate}
\item $v(uw-4)+12w(b-1)>0$,
\item $12w(b-1)>v(uw-4)>0$'
\item $\bigl(12w(b-1)+v(uw-4)\bigr)^2 > 12vw(uw-4)b^2q^2$,
\end{enumerate}
then the third moment of~$\M$ is infinite.
\end{proposition}

\proof The above conditions mean $2 \theta_a<bq$ and $ \theta_a^2-bq
\theta_a+b-1>0$, which implies that $ \theta_a<\alpha$ and
$\psi_\alpha$ has no fixed point. Suppose that~$\M$ has a finite third
moment~$t$. Then $t=\moment_3\M=\moment_3\T\M = \psi_{\alpha}(t)$,
which is not possible since $\psi_{\alpha}$ has no fixed point.

To be complete, one should prove that these conditions can be
fulfilled. Indeed, in the case of Galton-Watson (see
Section~\ref{examples}), with $b_2=b_1^2$ the three requirements are
\begin{eqnarray*}
&& b^4q^4+8b^2-12b+4 > 0\\
&& 16b^2-36b+20-b^4q^4>0\\
&&b^8q^8-12b^6(b-1)q^6+8b^4(b-1)(2b-1)q^4\\
&&\hspace{6em}+48b^2(b-1)^2(b-2)q^2+ 16(b-1)^2(2b-1)^2>0.
\end{eqnarray*}
The first inequality always holds (do not forget that $b>2$). If
\begin{equation*}
  q^4<
  \min\left\{\frac{16b^2-36b+20}{b^4},
    \frac{48b^2(b-1)^2(b-2)}{12b^6(b-1)}\right\} = \frac{4(b-1)(b-2)}{b^4}
\end{equation*}
the remaining inequalities are fulfilled.
\bigskip

The following facts, when $v(uw-4)+12w(b-1)>0$, easily result from
Remarks~\ref{obvious}.
\begin{enumerate}
\item If $\alpha\le \theta< \vartheta$ and $\gamma_{{}_-}(\theta)\le t\le
  \gamma_{{}_+}(\theta)$, then
\begin{equation}\label{fait1}
  \gamma_{{}_-}\bigl( \varphi(\theta)\bigr)\le \gamma_{{}_-}(\theta)\le
  \psi_\theta(t)\le t\le \gamma_{{}_+}(\theta)\le \gamma_{{}_+}\bigl(
  \varphi(\theta)\bigr). 
\end{equation}
\item If $\displaystyle \theta\le \alpha\le \theta_a$ and $t\le
  \gamma_{{}_-}(\theta)$, 
  then
\begin{equation}\label{fait2}
 t \le \psi_\theta(t)\le \gamma_{{}_-}(\theta)\le \gamma_{{}_-}\bigl(
 \varphi(\theta)\bigr). 
\end{equation}
\end{enumerate}
\medskip

Let us consider the following subsets of~${\mathbb R}^2$:
\begin{eqnarray*}
  \Omega_1 &=& \left\{(\theta,t)\ :\ \alpha\le \theta<\vartheta,\
    \gamma_{{}_-}(\theta)\le t\le 
    \gamma_{{}_+}(\theta)\right\},\\
% \noalign \text{and}
\Omega_2 &=& \left\{(\theta,t)\ :\ \theta\le \alpha,\ t\le
  \gamma_{{}_-}(\theta)\right\} \text{\quad if\quad} \alpha\le \theta_a.
\end{eqnarray*}
The set~$\Omega_1$ is invariant under~$\Phi$, and, if $\displaystyle
\alpha\le \theta_a$, so is $\Omega_2$ (notice that if $\alpha=1$ then
$\theta_a>1$ and $\Omega_2$ reduces to $\delta_1$).  \medskip

Set, for $j=1,\,2$,
\begin{equation*}
{\mathscr D}_j = \left\{\mu\in {\mathscr P}\ :\ \moment_1(\mu)=1,\,
  (\moment_2(\mu),\moment_3(\mu)\in \Omega_j)\right\}.
\end{equation*}

Then it follows from the above analysis that both these sets are
invariant under the transformation~$\mathsf T$. 

So, if $\mu\in {\mathscr D}_1\cup {\mathscr D}_2$, one has
$$\bigl(\moment_2(\T^n\mu),\moment_3(\T^n\mu) = \Phi^n( 
\moment_2(\mu),\moment_3(\mu))\bigr)$$
and
$$ \lim \moment_2
(\T^n \mu) = \alpha, \text{\quad and\quad} \lim \moment_3(\T^n\mu) =
\gamma_{{}_-}(\alpha).$$
\medskip

Of course this is of
interest only if ${\mathscr D}_1\cup {\mathscr D}_2$ is non-empty. In
particular, one has to 
take into account the inequalities $1\le \moment_2(\mu)^{1/2}\le
\moment_3(\mu)^{1/3}$.

Let us show that there are parameters such that ${\mathscr D}_1$ is
nonempty. Consider the Galton-Watson case with $q=1$. Then one has
$\alpha=1<\beta=b_1-1$, $\gamma_-(\alpha)=1$,
$\gamma_+(\alpha)=b_2-1>1$, and $\theta_a-1= \displaystyle
\frac{(b_2-2)^2(b_1-1)}{b_2^2+8b_2-12b_1+4}>0$. It results that for a
Galton-Watson process with $(q,b_1,b_2)$ in a neighborhood of
$(1,b,b^2)$, where $b$ in an integer larger than or equal to 3, the
set ${\mathscr D}_1$ is nonempty.

\section{The case $\mathbf{q=1}$: a central limit theorem}\label{sec5}

We still suppose that $\beta>1$ and~$W$ is a non-negative random variable.

In this case, we know that $\T^n\mu$ weakly converges towards the
Dirac mass at~1. We have the following result.

For $\mu\in \bigcup_{1<
    \gamma<\beta}{\mathscr P}_\gamma$ and $n\ge 1$, we
defined\ \ $\displaystyle\sigma_n=\left (\int (x-1)^2 \,\T^{ n}\!\mu
(\dif x)\right )^{1/2}$.

Then Equations~\eqref{mmnt4} and~\eqref{birap0} give
\begin{equation}\label{ecarttype}
  \sigma_{n+1}^{2} = \frac{\sigma_n^2}{b-1 - \sigma_n^2} \text{\quad
    and\quad} \frac{\sigma_n^2}{b-2-\sigma_n^2} =
  (b-1)^{-n}\,\frac{\sigma_0^2}{b-2-\sigma_0^2}. 
\end{equation}

\begin{lemma}\label{Bound3}
There exists $C$ such that for all non-negative $W$ whose distribution
is in ${\mathscr D}_1\setminus\{\delta_1\}$ one has
\begin{equation*}
  (u-\esp W^3)\, \esp Z_Y^3 \le (b-1)^{3/2} \esp Z_W^3 + C\bigl(
  (\esp Z_W^3)^{2/3} + (\esp Z_W^3)^{1/3} + 1\bigr),
\end{equation*}
where $Y=\T W$, $Z_W = |W-1|/\sigma_W$, and $Z_Y = |Y-1|/\sigma_Y$.
\end{lemma}

The proof, as well as that of the following corollary, follows the
same lines as in~\cite{BPW}. 

\begin{corollary}\label{bound3}
  If $(b-1)^3 < (u-1)^2$ and $\mu\in {\mathscr
    D}_1\setminus\{\delta_1\}$, then
$$\sup_n \int \sigma_n^{-3}|x-1|^3 \T^n\mu(\dif x) < \infty.$$
\end{corollary}

Now, we need to carefully iterate Formula~\eqref{foncn}.  We set
$\displaystyle Z_n = 
\frac{W_n-1}{\sigma_n}$. Equation~(\ref{foncn}) yields
\begin{equation}\label{c6}
  Z_{n+1} = \sum_{j\ge1} a_j\left[\sigma_n\,Z_{n}(j)\,
    Z_{n+1}(j) + \frac{\sigma_n}{\sigma_{n+1}}\, Z_{n}(j) + Z_{n+1}(j)
  \right].
\end{equation} 
If we set
\begin{equation}\label{R}
R_n = \sum_{j\ge1} a_jZ_{n}(j) Z_{n-1}(j) \sigma_{n-1}
+ \left( \frac{\sigma_{n-1}}{\sigma_n} - \sqrt{b-1}
\right) \sum_{j\ge1} a_j Z_{n-1}(j),
\end{equation}
then Equation~(\ref{c6}) rewrites as
\begin{equation}\label{c7}
  Z_{n+1} = R_{n+1} + \sum_{j\ge1}
  a_j Z_{n+1}(j) + \sqrt{b-1} \sum_{j\ge1} a_j Z_n(j).
\end{equation}

We are going to use repeatedly Formula~(\ref{c7}). Let~$\epsilon$
stand for empty word on any alphabet. For this purpose, fix~$n>1$, define
$R_n(\epsilon,\epsilon)=R_n$ as well as $Z_n(\epsilon,\epsilon)=Z_n$,
and write~(\ref{c7}) in the following way
\begin{equation}\label{c8}
  Z_{n}=Z_n(\epsilon,\epsilon) = R_{n}(\epsilon,\epsilon) + \sum_{j\ge1}
  a_j(\varepsilon,\varepsilon) Z_{n}(j,0) +\sqrt{b-1}
  \sum_{j\ge1}a_j(\varepsilon,\varepsilon) Z_{n-1}(j,1).
\end{equation}

Since we are interested in distributions only, we can take copies of
these variables so that we can write
\begin{eqnarray*}
  Z_{n}(j,0) &=& R_{n}(j,0) +  \sum_{k\ge1} a_k(j,0) Z_{n}(jk,00)
  +\sqrt{b-1} \sum_{k\ge1} a_k(j,0)
  Z_{n-1}(jk,01)\\
  Z_{n-1}(j,1) &=& R_{n-1}(j,1) +\sum_{k\ge1} a_k(j,1) Z_{n-1}(jk,10)\\
&&\hspace{8em}  {\ }+ \sqrt{b-1} \sum_{k\ge1} a_k(j,1) Z_{n-2}(jk,11).
\end{eqnarray*}

%**** add precisions 
Notice that since by definition in Formula~(\ref{c8}) the random
variables of the form $Z_{n-1}(j,w)$ and $Z_n(j,w)$ are mutually
independent and independent of~$a$, and the same holds for the random
variables $R_n(j,w)$ and $R_{n-1}(j,w)$, as well as for the random
variables $Z_{n-2}(jk,w)$, $Z_{n-1}(jk,w)$ and $Z_{n}(jk,w)$.

Then Formula~(\ref{c8}) rewrites as
\begin{eqnarray*}
Z_n(\epsilon,\epsilon) = R_n(\epsilon,\epsilon)&+& \sum_{j\in \alphabet}
a_j(\epsilon,\epsilon)\left( \sqrt{b-1}\,R_{n-1}(j,1) +
R_{n}(j,0)\right)\\
&+&\sum_{j,k\ge1} (b-1)a_j(\epsilon,\epsilon)Z_{n-2}(jk,11)\\
 &+& \sum_{j,k\ge1} \sqrt{b-1}\,a_j(\epsilon,\epsilon)a_k(j,1)Z_{n-1}(jk,10) \\
  &+&\sum_{j,k\ge1} \sqrt{b-1}a_j(\epsilon,\epsilon)a_k(j,0)Z_{n-1}(jk,01) \\
  &+&\sum_{j,k\ge1} a_j(\epsilon,\epsilon)a_k(j,0)Z_{n}(jk,00),
\end{eqnarray*}
and so on. At last we get $Z_n = T_{1,n} + T_{2,n}$, with
\begin{eqnarray}\label{T1n} 
  T_{1,n} 
  &=& \sum_{k=0}^{n-1}
  \sum_{\substack{m\in\{0,1\}^k\\w\in {\mathscr T}_k}}
  (b-1)^{\frac{k-\varsigma(m)}{2}}R_{n-k+\varsigma(m)}(w,m)
  \prod_{j=0}^{k-1} a_{w_{j+1}}(w|_j,m|_j)\label{r1}\\
  T_{2,n} &=& \sum_{\substack{m\in\{0,1\}^n\\w\in{\mathscr T}_n}}
  (b-1)^{\frac{n-\varsigma(m)}{2}} Z_{\varsigma(m)}(w,m) \prod_{j=0}^{n-1}
  a_{w_{j+1}}(w|_j,m|_j),\label{r2} 
\end{eqnarray}
where~$\varsigma(m)$ stands for the number of zeroes in~$m$.
Moreover, all variables in Equation~(\ref{r2}) are independent, and in
Equation~(\ref{r1}), the variables corresponding to the same~$k$ are
independent.
\smallskip

We can also use a more constructive approach to obtain the previous
decomposition of $Z_n$. At first, we notice that the meaning of
Equation~\eqref{c7} is the following: given independent variables
$Z_{n}(k)$ and $Z_{n+1}(k)$ (for $0\le k< b$) equidistributed with
$Z_{n}$ and $Z_{n+1}$, and independent of~$a$, if we define~$R_n$ by
Equation~\eqref{R}, then the right hand side of Equation~\eqref{c7}
is equidistributed with~$Z_{n+1}$.

Let~$n$ be fixed larger than~2. One starts with two
collections
$$
\bigl\{ Z_\ell(w,m)\bigr\}_{0\le \ell\le n,\,w\in {\mathcal T}_n,\, m\in
  \{0,1\}^n }\text{\quad and\quad}\bigl\{ a(w,m)\bigr\}_{0\le \ell\le
  n,\,w\in {\mathcal T}_\ell,\, m\in \{0,1\}^{\ell} }
$$
such that all these variables are independent, the $Z_\ell(\cdot,\cdot)$
have the same distribution as~$Z_\ell$, and the $a(\cdot,\cdot)$ have the
same distribution as~$a$.

One defines by recursion
\begin{multline*}
  R_\ell(w,m) = \sum_{j\ge1}  a_j(w,m)Z_{\ell-1}(wj,m0)
  Z_{\ell}(wj,m1) \sigma_{\ell-1} + {\ }\\ \left(
    \frac{\sigma_{\ell-1}}{\sigma_\ell} - \sqrt{b-1} \right)
  \sum_{j\ge1} a_j(w,m) Z_{\ell-1}(wj,m0)
\end{multline*}
and
\begin{multline*}
  Z_{\ell}(w,m) = R_{\ell}(w,m) + {\ }\\\sqrt{b-1}
  \sum_{j\ge1} a_j(w,m) Z_{\ell-1}(wj,m0) + 
  \sum_{j\ge1} a_j(w,m) Z_{\ell}(wj,m1),
\end{multline*}
for $0\le \ell\le n$, $(w,m)\in {\mathbb Z}_+^j\times \{0,1\}^j$ with~$j\ge
n-\ell$.
\medskip

Due to~\eqref{c7}, all these new variables $Z_\ell(\cdot.\cdot)$ are
equidistributed with~$Z_\ell$, and we get $Z_n(\epsilon,\epsilon) =
T_{1,n} + T_{2,n}$.

It will be convenient to denote by ${\mathscr A}_n$ the $\sigma$-field
generated by the variables~$a(w,m)$ with $|w|<n$ and $|m|<n$, and by
${\mathscr A}$ the $\sigma$-field generated by all the variables
$a(w,m)$.

\begin{proposition}\label{prop1}
We have $\displaystyle \lim_{n\to\infty} \esp{T_{1,n}^2} = 0$, so
$T_{1,n}$ converges in distribution to 0.
\end{proposition}

\begin{proof}  Set $r^2_n = \esp{R_n^2}$. We have
\begin{equation*}
b\, r_n^2 = \sigma_{n-1}^2+\left(
\frac{\sigma_{n-1}}{\sigma_n}-\sqrt{b-1}\right)^2,
\end{equation*}
which together with Formulae~(\ref{ecarttype}) implies that there
exists $C>0$ such that $r_n^2\le C^2(b-1)^{-n}$ for all $n\ge 1$. By
using the independence properties of random variables in~(\ref{T1n})
as well as the triangle inequality, we obtain
\begin{eqnarray*}
\bigl(\esp{T^2_{1,n}}\bigr)^{1/2}&\leq & \sum_{0\le k< n}  \left(
\sum_{|w|=|m|=k}
 (b-1)^{k-\varsigma(m)}r_{n-k+\varsigma(m)}^2\esp \prod_{j=0}^{k-1}
a_{w_{j+1}}^2\right)^{1/2}\\
&=&\sum_{0\le k< n}  \left(\sum_{j=0}^{k} \binom{k}{j}
\esp\left (\sum_{x\ge 1} a_x^2\right )^k(b-1)^{k-j}r_{n-k+j}\right)^{1/2}\\
&=& \sum_{0\le k< n}  \left(\sum_{j=0}^{k} \binom{k}{j}
b^{-k}(b-1)^{k-j}r_{n-k+j}\right)^{1/2}\\
&\le& C\sum_{0\le k< n} \left( \sum_{j=0}^{k} \binom{k}{j}
b^{-k}(b-1)^{k-j}(b-1)^{k-j-n}\right)^{1/2}\\
&=& C\sum_{0\le k< n}  b^{-k/2}(b-1)^{-n/2}\bigl( (b-1)^2+1\bigr)^{k/2}\\
&=&C\, (b-1)^{-n/2} \sum_{0\le k< n} \left
(\frac{(b-1)^2+1}{b}\right)^{k/2}\\ &=&\mathrm{O}\left( \Big
(1 - \frac{b-2}{b(b-1)}\Big )^{n/2}\right ).
\end{eqnarray*}
\end{proof}

\begin{lemma}\label{martingale1}
  $U_n = \esp(T_{2,n}^2\mid {\mathscr A})$ is a non-negative
  martingale. Denote its almost sure limit by $U$.
\end{lemma}

\proof
We have
\begin{equation*}
  U_n = \sum_{\substack{m\in\{0,1\}^n\\w\in{\mathscr T}_n}}
  (b-1)^{n-\varsigma(m)} \prod_{j=0}^{n-1} a_{w_{j+1}}^2(w|_j,m|_j)
\end{equation*}
and
\begin{eqnarray*}
  \esp(U_{n+1}&&\hspace{-2.5em}\mid {\mathscr A}_n)\\
&=& \sum_{\substack{m\in\{0,1\}^n\\w\in{\mathscr T}_n}}
  \prod_{j=0}^{n-1} a_{w_{j+1}}^2(w|_j,m|_j)
  \esp\sum_{\substack{k\in\{0,1\}\\x\ge 1}}
  (b-1)^{(n+1-\varsigma(mk))} a_x^2(w,m)\\
&=& \sum_{\substack{m\in\{0,1\}^n\\w\in{\mathscr T}_n}}
  \prod_{j=0}^{n-1} a_{w_{j+1}}^2(w|_j,m|_j) \sum_{k\in\{0,1\}}
  b^{-1}(b-1)^{(n+1-\varsigma(mk))}\\
&=& U_n.
\end{eqnarray*}

In fact, $U_n$ is a standard Mandelbrot multiplicative  martingale
built on the tree $\bigcup_{n\ge 1}  {\mathscr T}_n\times \{0,1\}^n$. Indeed,
for each $(w,m)\in \bigcup_{n\ge 1}  {\mathscr T}_n\times \{0,1\}^n$
define the vector
$A(w,m)=(A_{j,\epsilon}(w,m))_{(j,\epsilon)\in{\mathscr
    T}_1\times \{0,1\}}$, where $A_{j,0}(w,m)=a_{j}^2(w,m)$ and $A_{j,1}(m,w)=
(b-1)a_{j}^2(w,m)$. By construction we have
$\esp\sum_{(j,\epsilon)\in {\mathscr
    T}_1\times \{0,1\}}A_{j,\epsilon}(w,m)=1$, and  
$$
U_n= \sum_{\substack{m\in\{0,1\}^n\\w\in{\mathscr
      T}_n}}\prod_{j=0}^{n-1} A_{w_{j+1},m_{j+1}}(w|_{j},m|_{j}). 
$$

\begin{lemma}\label{martingale2} Let $\kappa>0$. Suppose that 
  $u_{\kappa} =
  \left(\esp\sum_{j\ge1}a_j^\kappa\right)^{-1}>0$. Then
$$V_{\kappa,n}= 
  \left(\frac{(b-1)^{\kappa/2}+1}{u_{\kappa}}\right)^{-n}
  \sum_{\substack{m\in\{0,1\}^n\\w\in{\mathscr T}_n}}
  (b-1)^{\frac{\kappa\bigl(n-\varsigma(m)\bigr)}{2}} \prod_{j=0}^{n-1}
  a_{w_{j+1}}^\kappa(w|_j,m|_j),$$
is a martingale.
\end{lemma}

\proof This results from a computation similar to the previous one, or
the observation that $V_{\kappa,n}$ is the Mandelbrot martingale
associated with the random vectors $
A_{\kappa}(m,w)=\left(\frac{(b-1)^{\kappa/2}+1}{u_{\kappa}}\right)^{-1}
(A^{\kappa/2}_{j,\epsilon}(w,m))_{(j,\epsilon)\in{\mathscr
    T}_1\times \{0,1\}}$. 

\begin{corollary}\label{smallterms} \begin{enumerate}
\item For $\kappa>0$, if $u_{\kappa}>0$,
\begin{equation*}
  \sup_{\substack{m\in\{0,1\}^n\\w\in{\mathscr T}_n}}
  (b-1)^{(n-\varsigma(m))} \prod_{j=0}^{n-1} a_{w_{j+1}}^2(w|_j,m|_j)
  \le \left(\frac{(b-1)^{\kappa/2}+1}{u_{\kappa}}\right)^{2n/\kappa}
  V_{\kappa,n}^{2/\kappa}. 
\end{equation*}
 \item If  $(b-1)^{\kappa/2}+1<u_{\kappa}$ for
   some $\kappa>0$, with probability~1 
 $$ \lim_{n\to \infty }\sup_{\substack{m\in\{0,1\}^n\\w\in{\mathscr
       T}_n}} (b-1)^{(n-\varsigma(m))} \prod_{j=0}^{n-1}
   a_{w_{j+1}}^2(w|_j,m|_j) = 0.$$
   \item If $\kappa>2$, $(b-1)^{\kappa/2}+1<u_{\kappa}$ and $\esp
     (\sum_{j\ge 1}a_j^2)^{\kappa/2}<\infty$, then $\mathbb P(U>0)>0$,
     and  $\mathbb P(U>0)=1$ if and only if $\mathbb P(\#\{j\ge 1:
     a_j>0\}\ge 1)=1$.  
\end{enumerate}
\end{corollary}

\proof Let~$V_\kappa$ be the a.s.\ limit of the non-negative
martingale~$V_{\kappa,n}$ of Lemma~\ref{martingale2}. Since~$V_\kappa$
is integrable, it is a.s.\ finite. So
\begin{equation*}
  \sup_{\substack{m\in\{0,1\}^n\\w\in{\mathscr T}_n}}
  (b-1)^{\kappa(n-\varsigma(m))/2} \prod_{j=0}^{n-1}  a_{w_{j+1}}^{\kappa}(w|_j,m|_j)
  \le \left(\frac{(b-1)^{\kappa/2}+1}{u_{\kappa}}\right)^n V_{\kappa,n}.
\end{equation*}
This accounts for the first and second assertion. 

For the third assertion, we notice that our assumptions are exactly
those required for the Mandelbrot martingale $U_n$ to be bounded in
$L^{\kappa/2}$, hence have a non degenerate limit: \\
$\esp\sum_{(j,\epsilon)\in{\mathscr
    T}_1\times \{0,1\}}A^{\kappa/2}_{j,\epsilon}(w,m)<1$ and
$\esp(\sum_{(j,\epsilon)\in{\mathscr
    T}_1\times \{0,1\}}A_{j,\epsilon}(w,m))^{\kappa/2}<\infty$. The assertion on the
possibility that $U_n$ vanishes is then standard.

\begin{theorem}\label{thm13}
If $(b-1)^3 < (u-1)^2$ and  $\mu\in {\mathscr
  D}_1\setminus\{\delta_1\}$, the sequence 
$\displaystyle \sigma_n^{-1}(W_n-1)$ converges in distribution to
%$B(U)$, where $B$ is a Brownian motion independent of~$U$, or equivalently 
$\sqrt{U}\xi$, where $\xi$ is a standard normal law independent of~$U$.
\end{theorem}

\proof At first we notice that the assumptions of
Corollary~\ref{smallterms} are satisfied with $\kappa=3$. Indeed
$(b-1)^3 < (u-1)^2$ is just $(b-1)^{3/2} +1<u_3$, and $\esp
(\sum_{j\ge 1} a_j^{2})^{3/2}\le \esp (\sum_{j\ge 1} a_j)^{3}<\infty$
due to our assumption on $u$, $v$, and $w$.  Then, due to
Proposition~\ref{prop1} if suffices to prove the same 
convergence for the sequence $(T_{2,n})_{n\ge 0}$.

We adapt a proof given by Breiman~\cite{breiman} for
Lindeberg's theorem. First we remark that, if $X$ is a centered random
variable with standard deviation~$\sigma$ and $t\in {\mathbb R}$, one
has $|\esp \e^{\mi tX}-1|\le \frac{t^2\sigma^2}{2}$, and, if $\esp
|X|^3$ is finite, $|\esp \e^{\mi tX}-1+\frac{\sigma^2t^2}{2}| \le
\frac{|t^3| \esp |X|^3}{6}$. Also, if $|z|\le 1/2$, $|\log (1+z)-z|
\le |z|^2$.

For~$n\ge 1$, let $A_n = \{V_{3,n}\le n\}$. Since $V_{3,n}$ has a
limit finite with probability~1, the variable $\ind_{A_n}$ converges
towards~1 with probability~1.

For $w\in {\mathscr T}_n$ and $m\in \{0,1\}^n$, set
$$f_{n,w,m}(t) = \esp \left( \e^{\mi t (b-1)^{(n-\varsigma(m))/2}
  Z_{\varsigma(m)}(w,m)\prod_{j=0}^{n-1} a_{w_{j+1}}(w|_j,m|_j)}
\Big\vert {\mathscr A}_n\right).$$
 According to Corollary~\ref{smallterms} applied with $\kappa=3$, for
 all~$t$, we have 
\begin{equation*}
\sup_{\substack{m\in\{0,1\}^n\\w\in{\mathscr T}_n}}|f_{n,w,m}(t)-1| \le t^2
\left( \frac{(b-1)^{3/2}+1}{u}\right)^{2n/3}V_{3,n}^{2/3}\le n^{2/3}t^2
\left( \frac{(b-1)^{3/2}+1}{u}\right)^{2n/3}
\end{equation*}
on~$A_n$. So, since $(b-1)^3<u^2$, $t$ being fixed, for~$n$ large
enough we have
\begin{equation*}
\sup_{\substack{m\in\{0,1\}^n\\w\in{\mathscr T}_n}}|f_{n,w,m}(t)-1| \le \frac{1}{2}
\end{equation*}
and therefore
\begin{equation*}
\abs{\log f_{n,w,m}(t)-(f_{n,w,m}(t)-1)} \le |f_{n,w,m}(t)-1|^2.
\end{equation*}
But
\begin{multline*}
\left|f_{n,w,m}(t)-1+\frac{t^2}{2}(b-1)^{(n-\varsigma(m))} \prod_{j=0}^{n-1}
a_{w_{j+1}}^2(w|_j,m|_j)\right|\\
\le \frac{|t|^3}{6} (b-1)^{3(n-\varsigma(m))/2} \prod_{j=0}^{n-1}
a_{w_{j+1}}^3(w|_j,m|_j)\sup_{j\ge 0} \esp |Z_j|^3.
\end{multline*}
So, if we set $\displaystyle g_n(t) =
\sum_{\substack{m\in\{0,1\}^n\\w\in{\mathscr T}_n}} \log
f_{n,w,m}(t)$\ \ and\ \ $C = \sup_{j\ge 0} \esp |Z_j|^3$, for
fixed~$t$, for~$n$ large enough, on $A_n$,
\begin{equation*}
  \left| g_n(t)+\frac{t^2U_n}{2}\right| \le
  \sum_{\substack{m\in\{0,1\}^n\\w\in{\mathscr T}_n}}
  |f_{n,w,m}(t)-1|^2 + C\,\left(\frac{(b-1)^{3/2}+1}{u}\right)^n
  |t|^3V_{3,n}.
\end{equation*}
By writing $\sum |f_{n,w,m}-1|^2 \le \bigl( \sup |f_{n,w,m}-1|\bigr)\sum
|f_{n,w,m}-1|$ one gets on 
$A_n$
\begin{multline*}
  \left|g_n(t) + \frac{t^2U_n}{2}\right| \le\\ t^4\left(
    \frac{(b-1)^{3/2}+1}{u}\right)^{2n/3}V_{3,n}^{2/3}U_n +
  C\,\left(\frac{(b-1)^{3/2}+1}{u}\right)^n |t|^3V_{3,n}.
\end{multline*}

We have obtained 
$$ \esp\left(e^{\mi t T_{2,n}}\right)=\esp\left
(e^{-\frac{t^2U_n}{2}+r_n(t)}\mathbf{1}_{\{V_{3,n}\le
  n\}}\right)+\esp\left(e^{\mi t T_{2,n}}\mathbf{1}_{\{V_{3,n}>
  n\}}\right),
$$ with $|r_n(t)|\le t^4\left(
\frac{(b-1)^{3/2}+1}{u}\right)^{2n/3}n^{2/3}U_n + C
\left(\frac{(b-1)^{3/2}+1}{u}\right)^n |t|^3 n$ on $A_n$.
  
  Since both $U_n$ and $V_n$ converge almost surely and
  $\frac{(b-1)^{3/2}+1}{u}<1$, we obtain $\lim_{n\to\infty}\esp\left
  (e^{-\frac{t^2U_n}{2}+r_n(t)}\mathbf{1}_{\{V_{3,n}\le
    n\}}\right)=\esp\left (e^{-\frac{t^2U}{2}}\right )$ and
  $\lim_{n\to\infty}\esp\left(e^{\mi t T_{2,n}}\mathbf{1}_{\{V_{3,n}>
    n\}}\right)=0$ by the bounded convergence theorem.

\section{Central limit theorem in the complex case}\label{complex}

In this section, we suppose~$q=1$ and $\beta>1$, and study the
convergence in law of 
$(b-1)^{n/2}\bigl( W_n-1\bigr))$ when $W_0$ is a complex valued random
variable.

The fact $q=1$ implies the relation $\displaystyle \frac{1}{u} +
\frac{3}{v} + \frac{1}{w} = 1.$

\subsection{Variances and covariances}\label{clt61}

According to~\eqref{mmnt4} and~\eqref{mmnt3}, if $\esp |W_0^2| < b-1$,
we have
\begin{equation*}
\esp |W_{n+1}^2| = \frac{b-1}{b-\esp |W_n^2|} \text{\quad and \quad}
\esp W_{n+1}^2 = \frac{b-1}{b-\esp W_n^2},
\end{equation*}
so, both $\esp |W_{n+1}^2|$ and $\esp W_{n+1}^2$ converge to the fixed
point~1 of $\varphi$. Due to~\eqref{birap}, $(b-1)^n\bigl( \esp
W_n^2-1\bigr)$ and $(b-1)^n\bigl( \esp |W_n^2|-1\bigr)$ have explicit limits
when~$n$ goes to~$\infty$.

It results that, if we set
$$x_n = \esp (\Re W_n)^2,\quad y_n = \esp \bigl( \Im W_n\bigr)^2,
\text{\quad and \quad} z_n = \esp (\Re W_n)(\Im W_n)=\esp (\Re
W_n-1)(\Im W_n)$$ there exist~$x$,~$y$, and~$z$ (depending on $W_0$)
such that
\begin{equation}\label{covariances}
  \lim\, (b-1)^n (x_n-1) = x,\quad \lim\, (b-1)^n y_n =y, \text{\quad and
    \quad}\lim\, (b-1)^n z_n = z.
\end{equation}

Set $u=\esp (\Re W_0)^2=\esp (\Re W_0-1)^2+1$, $v=\esp (\Im W_0)^2$
and $w=\esp (\Re W_0)(\Im W_0)=\esp (\Re W_0-1)(\Im W_0)$. Notice that
by Cauchy-Schwarz inequality we must have $w^2\le v(u-1)$.  Using a
formal computing software (e.g. Maple) shows that one has
$$ (b-2)^{-3}\, \mathrm{det}\, \begin{pmatrix}x&z\\z&y\end{pmatrix} =
  \frac{v(bu-u^2+v^2-b+1)-(b-2u-2v)w^2}{(b-1-u-v)^2\bigl(4w^2+(b-u+v-1)^2
    \bigr)}.
$$ 

It is easily seen that the denominator of this last expression is
positive. Its numerator, call it $R(u,v,w)$, assumes its minimum
for~$w=0$ if $u+v\ge 2b$, and for $w^2=v(u-1)$ otherwise. We have
\begin{equation*}
R(u,v,0) = \frac14\,v(b-2)^3\bigl((2v-2u-b)(2v+2u+b)+ (b-2)^2\bigr)
\end{equation*}
and
\begin{equation*}
R(u,v,\pm\sqrt{v(u-1)}) = v(u+v-1)^2(b-2)^3.
\end{equation*}

It results that this determinant is positive, except if~$v=0$. We thus
have proven the following proposition.

\begin{proposition}
The matrix $\begin{pmatrix}x&z\\z&y\end{pmatrix}$ is definite positive
  if and only if and only if $W_0$ is not almost surely real.
\end{proposition}

\subsection{Central limit theorem}\label{clt62} 

In this context, if $Y=\T W$, we have
\begin{equation*}
\esp Y^3 = \esp W^3\esp Y^3/u+3\esp W^2\esp Y^2/v + 1/w
\end{equation*}
and
\begin{eqnarray*}
\esp |Y|^3 &\le& \esp |W|^3\esp |Y|^3/u+3\esp
|W|^2\esp |Y|^2\esp |W|\esp |Y|/v + \bigl( \esp |W|\esp
|Y|\bigr)^3/w\\
&\le& \frac{1}{u}\esp |W|^3\esp |Y|^3 +\frac{u-1}{u}\bigl( \esp
|W|^2\esp |Y|^2\bigr)^{3/2}. 
\end{eqnarray*}
Finally 
\begin{equation}\label{majmnt3}
  \esp |Y|^3 \le \frac{(u-1)\bigl( \esp |W|^2\esp
    |Y|^2\bigr)^{3/2}}{u-\esp |W^3|}. 
\end{equation}

This time for $1\le \theta \le \beta$ we use the function
\begin{equation}\label{newpsi}
\psi_\theta(t) = \frac{(u-1)\bigl( \theta\varphi(\theta)\bigr)^{3/2}}{u-t}.
\end{equation}
It has fixed points if $\theta\varphi(\theta)\le \left(
  \frac{u^2}{4(u-1)}\right)^{2/3}$, i.e., if $\theta\le \theta_a$ for
some critical real number $\theta_a$. As $\psi_1$ has two fixed 
points~1 and~$u-1$, one has $\theta_a>1$.
If $\theta\le \theta_a$, let $\gamma_\pm(\theta)$ be the fixed
points. Consider the sets
$$
\Omega = \left\{(\theta,t)\ :\theta< \min(\beta,\theta_a), t\le
  \gamma_{{}_+}(\theta)\right\},
  $$ 
 and 
\begin{equation*}
{\mathscr D}= \left\{\mu\in {\mathscr P}\ :\ \moment_1(\mu)=1,\,
  (\moment_2(\mu),\moment_3(\mu)\in \Omega)\right\}.
\end{equation*} Arguments similar to previous ones show
that if $W_0$ 
is distributed according to $\mu\in{\mathscr D}$, then, $\esp |W_n|^3<
\infty$ for all~$n$, and $\limsup \esp |W_n|^3\le 1$. Since $\esp
|W_n|^3\ge 1$, we have 
\begin{equation*}
  \lim \esp |W_n|^3 = \limsup \esp |W_n|^3 = 1.
\end{equation*}

Lemma~\ref{Bound3} and its corollary have the following counterparts.

\begin{lemma}
  There exists $C$ such that for all $W$ whose distribution is in
  ${\mathscr D}\setminus\{\delta_1\}$ one has
\begin{equation*}
  (u-\esp |W|^3)\, \esp Z_Y^3 \le (b-1)^{3/2} \esp Z_W^3 + C\bigl(
  (\esp Z_W^3)^{2/3} + (\esp Z_W^3)^{1/3} + 1\bigr),
\end{equation*}
where $Y=\T W$, $Z_W = |W-1|/\sigma_W$, and $Z_W = |Y-1|/\sigma_Y$.
\end{lemma}

\begin{corollary}
  If $(b-1)^3 < (u-1)^2$ and $\mu\in {\mathscr
    D}\setminus\{\delta_1\}$, then
$$\sup_n \int (b-1)^{3n/2}|x-1|^3 \T^n\mu(\dif x) < \infty.$$
\end{corollary}

Estimates similar to those used in the non-negative case now yield.

\begin{theorem}\label{thm16}
  If $(b-1)^3 < (u-1)^2$ and the distribution of~$W_0$ lies
  in~${\mathscr D}\setminus\{\delta_1\}$,
%  $(b-1)^{n/2}\bigl(W_n-1\bigr)$ converges in law to 
%  $B(U)$,
%  where $B$ is independent of $U$, and is a realization of the unique
%  (in law) Gaussian process starting at $0$, with independent and
%  stationary increments, and covariance matrix given by
%  $\begin{pmatrix}x&z\\z&y\end{pmatrix}$ at~$1$ (where $x,\,y$, and
%  $z$ are defined in \eqref{covariances}). Equivalently, 
  $(b-1)^{n/2}\bigl(W_n-1\bigr)$ converges in law to $\sqrt{U} \xi$,
  where $\xi$ is a centered normal vector independent of $U$ whose
  covariance matrix is~$\begin{pmatrix}x&z\\z&y\end{pmatrix}$.
  \end{theorem}

\section{Higher order smoothing transformations}

Now, we are given a non-negative integer valued random variable~$N$
such that the radius of convergence~$R$ of its probability generating
function~$f(t) = \sum_{\nu\ge 0} \Proba(N=\nu)\,t^\nu$ is larger
than~1. 

If $W$ is a square integrable random variable of
expectation~1, we define $\tilde W$ to be the
product of $N$ independent random variables equidistributed
with~$W$ and independent of~$N$: 
\begin{equation}\label{Nprod}
\tilde W = \prod_{1\le \nu\le N} W_\nu.
\end{equation}

Then we perform the cascade construction, as in
Section~\ref{cascades}, but, this time, using~$\tilde W$ instead
of~$W$. It is easy to see that the corresponding martingale is bounded
in $L^2$ if and only if $f(\esp W^2)< bq$ and, if it is so, the
limit~$Y$ of this martingale satisfies
\begin{equation*}
  \esp Y^2 = \frac{b-1}{bq - f(\esp W^2)} \text{\quad and\quad}
  \esp |Y|^2 = \frac{b-1}{bq - f(\esp |W|^2)}. 
\end{equation*}

Let~$\T_N$ stand for the map which sends the distribution of~$W$
to the one of~$Y$.

To iterate this operation, this time we have to deal with the function
\begin{equation*}
\varphi(x) = \frac{b-1}{bq-f(x)}.
\end{equation*}

In the interval $[0,f^{-1}(bq)]$ this mapping has at most two positive
fixed points, the roots of the 
function $p(x) = xf(x)-bqx+b-1$. Observe that $p(0)=
p\bigl(f^{-1}(bq)\bigr)=b-1>0$. Then, $\varphi$ has two fixed
points~$\alpha$ and~$\beta$ such that 
$0\le\alpha\le\beta< f^{-1}(bq)$ if 
and only if the minimum of $p$ on this interval is non-positive. This
happens if and only if $x_0^2f'(x_0)\ge b-1$, where $x_0$ is the
solution to equation $x_0f'(x_0)+f(x_0)=bq$. But, as previously,
we wish that $\alpha\ge 1$. This means $x_0\ge 1$, i.e.,
$f'(1)+f(1)\le bq$. As $0<q\le 1$, this gives $b\ge 1+\esp N$. When
this condition is fulfilled, then $q$ is subject to the
restriction $$\frac{1+\esp N}{b}\le q\le 1.$$
\medskip

From now on we suppose that
\begin{equation}\label{cond-n}
b > 1+\esp N \text{\quad and\quad} \frac{1+\esp N}{b}\le q\le 1
\end{equation}
(as previously we discard the trivial case $b=1+\esp N$ which yields
$q=\alpha=\beta=1$). 

\subsection{Fixed points} 

We also define the following non linear smoothing transformation~$\smooth_{N+1}$
which associates with a 
probability measure $\mu$ on $\mathbb C$ the measure
\begin{equation}\label{smooth-n}
  \smooth_{N+1}(\mu) = {\mathscr L}\left(\sum_{j\ge 1} a_j
    \prod_{k=0}^N W_k(j)\right), 
\end{equation}
where the random variables $\{W_k(j)\}_{j\ge 1,k\ge0}$ are distributed
according to~$\mu$, independent and independent of~$a$. 

Then, as previously, a fixed point of $\smooth_{N+1}$ is also a fixed
point of~$\T_N$.

In the same way, a fixed point of $\smooth_{N+1}$ is constructed as
the law of the limit of a martingale:

Consider the Galton-Waston tree ${\mathcal T}$ defined by the
variable~$N+1$. Let
${\mathcal T}_n$ stand for the nodes of generation~$n$. Consider 
$\{a(w,m)\}_{\substack{k\ge 1\\(w,m)\in  {\mathscr 
      T}_k\times {\mathcal T}_k}}$ a sequence of independent copies
of~$a$ and $\{N(w,m)\}_{\substack{k\ge 1\\(w,m)\in  {\mathscr 
      T}_k\times {\mathcal T}_k}}$ a sequence of independent copies
of~$N$ also independent of~$a$ and of~$\{a(w,m)\}$. For all 
$(w,m)$  we define 
$Z_1(w,m)=\sum_{j\ge 1} a_j(w,m)$. It has $\smooth_{N+1}(\delta_1)$
as distribution. Then we define recursively, for all $k\ge 2$ and
$(w,m)$,
\begin{equation*}
Z_k(w,m)=\sum_{j\ge 1} a_j(w,m) \prod_{\ell=0}^{N(w,m)} Z_{k-1}(wj,m\ell),
\end{equation*}
which by induction is clearly distributed according to
$\smooth_{N+1}^k(\delta_1)$.

Set $Z_k=Z_k(\epsilon,\epsilon)$. As previously, the sequence
$(Z_k)_{k\ge 1}$ is a non-negative martingale. The law $\M$ of its
limit is a fixed point of $\smooth_{N+1}$ as well as $\T_N$ and the
same analysis as in Section~\ref{DYN} can be performed: if
\eqref{cond-n} holds, $\M_\nu$ is the unique fixed point of $T_{\nu}$
belonging to $\mathscr{P}_\alpha\cap {\mathcal P}({\mathbb R}^+)$, and
for all $\mu\in {\mathcal P}({\mathbb R}^+)\cap\bigcup_{\alpha\le
  \gamma<\beta}{\mathscr P}_\gamma$ if $\beta>\alpha$ and all $\mu\in
{\mathcal P}({\mathbb R}^+)\cap{\mathscr P}_\alpha$ if $\beta=\alpha$,
the sequence $\T_N^{j}\mu$ converges to~$\M$. Moreover, if the
$W_{j,k}$ are independent and independent of~$a$ and~$N$, and all
distributed according to~$\M$, then
\begin{equation*}
W = \sum_{j\ge 1} a_j \prod_{k=0}^{N} W_{j,k}
\end{equation*}
is distributed according to~$\M$.

If we wish to deal with measures not supported on~${\mathbb R}^+$, we
have to make the extra assumption that $\esp \overline W_0W_0'$ lies
in the basin of attraction of the fixed point~$\alpha$ to get the
analog of Lemma~\ref{sim}. This assumption is automatically fulfilled
when $N$ is the constant~1. With $\Proba(N=2)>0$, the result holds
without this assumption if we restrict ourselves to probability
supported in~${\mathbb R}$. But in general we do not know whether it
may happen that the mapping~$\varphi$ has other attractive fixed
points.

\subsection{Central limit theorem}

In this section we suppose that $q=1$ and still assume \eqref{cond-n}
holds. We just outline modifications to be brought to
Sections~\ref{clt61} and~\ref{clt62}. When starting from $\mu\in
\bigcup_{1< \gamma<\beta}{\mathscr P}_\gamma$, one has $\lim \esp
|W_n|^2=1$ and, since $\esp W_n=1$, $\lim \esp W_n^2=1$.  Since $ \esp
N/(b-1)=\varphi'(1)$, the following limits exist
\begin{equation*}
  \lim_{k\to \infty} \left( (\esp N)^{-1}(b-1)\right)^k\bigl(\esp
  W_k^2-1\bigr) \text{\quad \and\quad} \lim_{k\to \infty} \left( (\esp
  N)^{-1}(b-1)\right)^k\bigl(\esp |W_k|^2-1\bigr).
\end{equation*}

It results that, if we set
$$x_n = \esp (\Re W_n)^2,\quad y_n = \esp \bigl( \Im
W_n\bigr)^2, \text{\quad and \quad} z_n = \esp (\Re W_n)(\Im W_n)$$
there exist~$x$,~$y$, and~$z$ (depending on $W_0$) such that
\begin{equation*}
  \lim\, \frac{ x_n-1}{\big ((\esp N)^{-1}(b-1)\big )^n} = x,\quad
  \lim\, \frac{ y_n}{\big ((\esp N)^{-1}(b-1)\big )^n}  =y,  
  \text{\quad and
    \quad}\lim\, \frac{ z_n-1}{\big ((\esp N)^{-1}(b-1)\big )^n} = z.
\end{equation*}

If $Y=\T_N W$, Formulae~\eqref{Recur3}, \eqref{majmnt3},
and~\eqref{newpsi} become 
\begin{eqnarray*}
  \esp Y^3 &=& \frac{u\bigl(3wf(\esp W^2)\esp Y^2+v\bigr)}{vw\bigl( u-f(\esp
    W^3)\bigr)},\\
\esp |Y|^3 &\le& \frac{(u-1)\bigl( f(\esp |W^2|)\esp
  |Y|^2\bigr)^{3/2}} {u-f(\esp |W|^3)},\\
\psi_\theta(t) &=& \frac{(u-1)\bigl(
  f(\theta)\varphi(\theta)\bigr)^{3/2}}{u-f(t)}.
\end{eqnarray*}
The function $\psi_\theta$ has two fixed points between~0 and~$u$ if
and only if $\theta\le \theta_{a,N}$ for some critical real number
$\theta_{a,N}$. But as $\psi_1$ has two fixed points~1 and~$u-1$, one
has $\theta_{a,N}>1$. If $1\le \theta\le \theta_{a,N}$, let
$\gamma_\pm(\theta)$ be the fixed points. Consider the sets
$$
\Omega = \left\{(\theta,t)\ :\theta< \min(\beta,\theta_{a,N}), t\le
  \gamma_{{}_+}(\theta)\right\},
  $$ 
 and 
\begin{equation*}
{\mathscr D}= \left\{\mu\in {\mathscr P}\ :\ \moment_1(\mu)=1,\,
  (\moment_2(\mu),\moment_3(\mu)\in \Omega)\right\}.
\end{equation*}

Arguments similar to previous ones show that if $W_0$ is distributed
according to $\mu\in{\mathscr D}$, then, $\esp |W_n|^3< \infty$ for
all~$n$, and $\limsup \esp |W_n|^3\le 1$. Since $\esp |W_n|^3\ge 1$,
we have
\begin{equation*}
  \lim \esp |W_n|^3 = \limsup \esp |W_n|^3 = 1.
\end{equation*}

Lemma~\ref{Bound3} and its corollary have the following counterparts.

\begin{lemma}
There exists $C$ such that for all $W_0$ whose distribution
is in ${\mathscr D}\setminus\{\delta_1\}$ one has
\begin{multline*}
  \bigl(u-f(\esp |W_{n}|^3)\bigr)\, \esp Z_{W_{n+1}}^3 \le
  C_n(\esp(N))^{-1/2}(b-1)^{3/2} \esp\bigl( N^3(\esp |W|^3)^{N-1}\bigr)\esp
  Z_{W_n}^3\\ + C\bigl( (\esp Z_{W_n}^3)^{2/3} + (\esp
  Z_{W_n}^3)^{1/3} + 1\bigr),
\end{multline*}
where $W_{n}=\T_N^n W_0$, $Z_{W_n} = |W_n-1|/\sigma_{W_n}$, and $C_n =
1+{\mathrm o}(1)$.
\end{lemma}

\begin{corollary}\label{cor18}
  If $(b-1)^{3/2}(\esp(N))^{-1/2} < u-1$ and $\mu\in {\mathscr
    D}\setminus\{\delta_1\}$, then
%$$\sup_n \int ((\esp N)^{-1}(b-1))^{3n/2}|x-1|^3 \T^n\mu(\dif x) <
%\infty.$$
$$\sup_n \int \sigma_n^{-3}|x-1|^3 \T^n\mu(\dif x) < \infty,$$
where $\sigma_n$ is the standard deviation of variables distributed
according to $\T_N(\mu_0)$.
\end{corollary}

We now have to mention the analog of the discussion following
Equation~\eqref{c6}. 

By iteration, we get a sequence $\{W_n\}$ of variables. We set
$\sigma_n^2 = \var W_n$, $\tilde{\sigma}_n^2 = \var \tilde{W}_n$, $Z_n
= \sigma_n^{-1}(W_n-1)$, and $\tilde{Z}_n =
\tilde{\sigma}_n^{-1}(\tilde{W}_n-1)$.

The following facts are easily proven:
\begin{itemize}
\item $\displaystyle \frac{\sigma_{n+1}}{\sigma_n} \sim
  \frac{\sqrt{\esp N}}{\sqrt{b-1}}$,
\item $\tilde{\sigma}_n^2 = f(\sigma_n^2+1)-1 \sim \sigma_n^2\esp N$,
\end{itemize}

Equation~\eqref{c6} becomes
\begin{equation}\label{c6n}
  Z_{n+1} = \sum_{j\ge1} a_j\left[\tilde{\sigma}_n\,\tilde{Z}_{n}(j)\,
    Z_{n+1}(j) + \frac{\tilde{\sigma}_n}{\sigma_{n+1}}\,
    \tilde{Z_{n}}(j) + Z_{n+1}(j) \right].
\end{equation} 

But as $\tilde{W}-1 = (W_1-1)W_2\cdots W_N+(W_2-1)W_3\cdots W_N+\cdots+
(W_N-1)$ (with $W_1,\,W_2,\cdots$ i.i.d.) we have $\tilde{W}-1 =
\sum_{1\le j\le N} W_j-1$ with a $L^2$-error of the same order of
magnitude as $\var W$. So Equation~\eqref{c6n} can be rewritten as 
\begin{equation}\label{c7n}
  Z_{n+1} = R_{n+1} + \sum_{j\ge1}
  a_j Z_{n+1}(j) + \frac{\sqrt{b-1}}{\sqrt{\esp N}} \sum_{j\ge1} a_j
  \sum_{1\le \nu\le N}Z_n(j,\nu), 
\end{equation}
where $R_{n+1}$ is a sum of 'error terms'.

As previously, we iterate this formula, but using a Galton-Waston tree
associated with the variable $N+1$ instead of using a binary tree.
Finally, we get
 $Z_n = T_{1,n} + T_{2,n}$, with
\begin{eqnarray}\label{T1nn} 
  T_{1,n} 
  &=& \sum_{k=0}^{n-1}
  \sum_{\substack{m\in{\mathcal T}_k\\w\in {\mathscr T}_k}}
\left( \frac{b-1}{\esp N}\right )^{\frac{k-\varsigma(m)}{2}}R_{n-k+\varsigma(m)}(w,m)
  \prod_{j=0}^{k-1} a_{w_{j+1}}(w|_j,m|_j)\label{r1n}\\
  T_{2,n} &=& \sum_{\substack{m\in {\mathcal T}_n\\w\in{\mathscr T}_n}}
 \left( \frac{b-1}{\esp N}\right )^{\frac{n-\varsigma(m)}{2}}
 Z_{\varsigma(m)}(w,m) \prod_{j=0}^{n-1} 
  a_{w_{j+1}}(w|_j,m|_j),\label{r2n} 
\end{eqnarray}
where~${\mathcal T}_n$ stands for the $n$-th generation nodes of the
Galton-Waston tree and $\varsigma(m)$ stands for the number of zeroes in~$m$.
Moreover, all variables in Equation~(\ref{r2n}) are independent, and in
Equation~(\ref{r1n}), the variables corresponding to the same~$k$ are
independent.\medskip

Then, arguing as previously yields the convergence in distribution of
$ ((\esp N)^{-1}(b-1))^{n/2}\bigl(W_n-1\bigr)$ as in
Section~\ref{complex}, with $U$ the limit of the Mandelbrot
multiplicative martingale built on the tree $\bigcup_{n\ge 1} {\mathscr T}_n\times \mathcal
T_n$ with the random vectors
$A(w,m)=(A_{j,\epsilon}(w,m))_{(j,\epsilon)\in
  {\mathscr T}_1\times \{0,1,\cdots,N(w,m)\}}$, where $A_{j,0}(w,m)=a^2_j(w,m)$,
$A_{j,\epsilon}(w,m)= (\esp N)^{-1}(b-1)a^2_j(w,m)$ for $1\le
\epsilon\le N(w,m)$, and the $((a_j(w,m))_{j\ge 1},N(w,m))$, $(w,m)\in
\bigcup_{n=0}^\infty \mathscr T_n\times\mathcal T_n$, are independent
copies of $((a_j)_{j\ge 1},N)$.

\section{A functional central limit theorem in the quadratic case}\label{sec8}

We suppose that $N=1$ almost surely and that~$q=1$. We are going to
use words on two alphabet. The ones, denoted by $w,\,v\,v'\dots$ are
finite sequences of positive integers. The others, denoted by $m,\,m'$
are finite sequences of $0$ and~$1$. It will be convenient to denote
$0^j$ the word composed of $j$ zeroes. Also, the concatenation will be
either denoted simply by juxtaposition or by a dot. The  expression
$v.w|_{j}$ should be understood as $v.(w|_j)$.

\subsection{Another writing for the martingale $(U_n)_{n\ge 1}$ and
  its limit}

We assume that $q=1$  and that the Mandelbrot martingale
$(U_n)_{n\ge 1}$ of the Section~\ref{sec5} converges almost surely and in $L^1$ to
its limit, i.e., is non degenerate. This is the case for instance under
the assumptions of Theorems~\ref{thm13} or~\ref{thm16}. We have
\begin{equation}\label{Udecomp}
U_n = \sum_{|w|=|m|=n} (b-1)^{n-\varsigma(m)} \prod_{0\le j< n}
a^2_{w_{j+1}}(w|_j,m|_j).
\end{equation}

We also consider the following quantities of the same kind: for $w$
and $m$ with the same length, we set
\begin{equation*}
U_n(w,m) = \sum_{|w'|=|m'|=n} (b-1)^{n-\varsigma(m')} \prod_{0\le j< n}
a^2_{w'_{j+1}}(w\cdot w'|_j,m\cdot m'|_j).
\end{equation*}

In Formula~\eqref{Udecomp}, we split the summation according to the
length of the prefix of maximal length of the form $0^k$ of~$m$. We
obtain
\begin{multline*}
U_n = \sum_{|v|=n} \prod_{0\le j< n} a^2_{v_{j+1}(v|_i,0^j)} + {}\\
\sum_{\substack{0\le k< n\\|v|=n\\|m'|=n-k-1}}\!\!\!\!
(b-1)^{n-k-\varsigma(m')}\!\! \prod_{0\le j\le k}
a^2_{v_{j+1}}(v|_j,0^j)\!\!\prod_{0\le
  j<n-k-1}\!a^2_{v_{k+j+2}}(v|_{j+k+1},0^k1\cdot m'|_j).
\end{multline*}

If we write $v=w\cdot w'$ with $|w|=k+1$, the second term of the right
hand side of the last formula rewrites as
\begin{multline*}
\sum_{0\le k< n} \sum_{|w|=k+1} (b-1)\prod_{0\le j\le
  k}a^2_{w_{j+1}}(w|_j,0^j)\times\\
\sum_{|w'|=|m'|=n-k-1} (b-1)^{n-k-1-\varsigma(m')} \prod_{0\le j<
  n-k-1} a^2_{w'_{j+1}}(w\cdot w'|_j,0^k1\cdot m'|_j).
\end{multline*}

Finally, we get
\begin{eqnarray*}
U_n &=& \sum_{|w|=n} \prod_{0\le j< n} a^2_{w_{j+1}(w|_j,0^j)}\\
&\hphantom{=}& + \sum_{0\le k< n}\sum_{|w|=k+1} (b-1)\Bigl(\prod_{0\le j\le
  k} a^2_{w_{j+1}}(w|_j,0^j) \Bigr) U_{n-k-1}(w,0^k1).
\end{eqnarray*}
It is worth noticing that the variables $U_{n-k-1} (w,0^k1 m)$, are
independent.

%% $$ U_n = \sum_{\substack{m\in {\mathcal T}_n\\w\in{\mathscr T}_n}}
%% \left( \frac{b-1}{\esp N}\right )^{n-\varsigma(m)} \prod_{j=0}^{n-1}
%% a^2_{w_{j+1}}(w|_j,m|_j).
%%   $$ Partitioning the tree ${\mathcal T}_n$ into the set $S_k$, $0\le
%% k\le n$, of words $w$ such that $w_{|k}=0^{k}$, and $w_{k+1}\neq 0$,
%% where $0^{k}$ stands for the word consisting of $k$ zeroes, we get
%% \begin{eqnarray*}
%% U_n&=&\sum_{k=0}^{n-1} \sum_{v\in {\mathscr T}_k} \sum_{\substack{1\le
%%     m\le N(v,0^{(k)})\\i\ge 1}} \left (\prod_{j=0}^{k-1}
%% a^2_{v_{j+1}}(v|_j,0^{(j)})\right) \left( \frac{b-1}{\esp N}\right )
%% a^2_i(v,0^{(k)})\, U_{n-k-1} (vi,0^{(k)}\cdot m)\\ &&+\sum_{v\in
%%   {\mathscr T}_n} \prod_{j=0}^{n-1} a^2_{v_{j+1}}(v|_j,0^{(j)}),
%% \end{eqnarray*}
%% where the random variables $U_{n-k-1} (vi,0^{(k)}\cdot m)$, are
%% independent copies of $U_{n-k-1}$ defined as
%% $$ U_{n-k-1} (vi,0^{(k)}\cdot m)= \sum_{\substack{m'\in {\mathcal
%%       T}_{n-k-1}\\w'\in{\mathscr T}_{n-k-1}}} \left( \frac{b-1}{\esp
%%   N}\right )^{n-k-1-\varsigma(m')} \prod_{j=0}^{n-k-2}
%% a^2_{w_{j+1}}(vi\cdot w|_j,0^{(k)}m\cdot m'|_j),
%%   $$
%% and with the convention $U_0=1$.   

As $U_k$ converges to $U$ almost surely and in $L^1$, and $\esp
\sum_{i\ge 1}a_i^2<1$, it follows that
%% $$ U=\sum_{k=0}^{\infty} \sum_{v\in {\mathscr T}_k} \Big
%% (\prod_{j=0}^{k-1} a^2_{v_{j+1}}(v|_j,0^{(j)})\Big ) \left(
%% \frac{b-1}{\esp N}\right )\sum_{\substack{1\le m\le N(v,0^{(j)})\\i\ge
%%     1}} a^2_i(v,0^{(k)})\, U (vi,0^{(k)}\cdot m) .
%% $$
\begin{equation*}
U=\sum_{k\ge 0} \sum_{|w|=k+1} (b-1) \Bigl(\prod_{0\le j\le k}
a^2_{w_{j+1}}(w|_j,0^j) \Bigr) U_{n-k-1}(w,0^k1).
\end{equation*}

Indeed, denoting by $U'$ the right hand side in the above inequality, we have 
\begin{multline*}
\|U'-U_n\|_1\le (b-1)\sum_{k=1}^n \Big(\esp \sum_{i\ge 1}a_i^2\Big)^k
\|U-U_{n-k}\|_1\\ +(b-1)\sum_{k\ge n} \Big(\esp\sum_{i\ge
  1}a_i^2\Big)^k+\Big(\esp\sum_{i\ge 1}a_i^2\Big)^n ,
\end{multline*}
hence by dominated convergence $\|U'-U_n\|_1$ tends to 0 as
$n\to\infty$, so that $U=U'$ almost surely.

Notice that all the copies of $U$ invoked in $U'$ are independent, so
the above relation rewrites
$$ U=(b-1) \sum_{k=1}^{\infty} \sum_{v\in {\mathscr T}_k}\left
(\prod_{j=0}^{k-1} a^2_{v_{j+1}}(v|_j)\right)\, U (v),
$$ where $(U (v))_{v\in \in\bigcup_{n\ge 1}\mathbb N_+^n)}$ is a
family of independent copies of $U$, which is independent of the
family of independent copies of $a$, $(a(w))_{w\in\bigcup_{n\ge
    0}\mathbb N_+^n}$.

\begin{remark}
In the general case (i.e., $N$ is not identically~$1$) we have 
%$$ U_n = \sum_{\substack{m\in {\mathcal T}_n\\w\in{\mathscr T}_n}}
%\left( \frac{b-1}{\esp N}\right )^{n-\varsigma(m)} \prod_{j=0}^{n-1}
%a^2_{w_{j+1}}(w|_j,m|_j).  $$ Partitioning the tree ${\mathcal T}_n$
%into the set $S_k$, $0\le k\le n$, of words $w$ such that
%$w_{|k}=0^{k}$, and $w_{k+1}\neq 0$, where $0^{k}$ stands for the
%word consisting of $k$ zeroes, we get
\begin{multline*}
U_n=\\
\sum_{k=0}^{n-1}\!\!
\sum_{\substack{v\in {\mathscr T}_k\\1\le m\le N(v,0^{(k)})\\i\ge 1}} \!\!
\left (\prod_{j=0}^{k-1}
a^2_{v_{j+1}}(v|_j,0^{(j)})\right) \left( \frac{b-1}{\esp N}\right )
a^2_i(v,0^{(k)})\, U_{n-k-1} (vi,0^{(k)}\cdot m)\\
+\sum_{v\in
  {\mathscr T}_n} \prod_{j=0}^{n-1} a^2_{v_{j+1}}(v|_j,0^{(j)}),
\end{multline*}
where the random variables $U_{n-k-1} (vi,0^{(k)}\cdot m)$, are
independent copies of $U_{n-k-1}$.
%%  defined as
%% $$ U_{n-k-1} (vi,0^{(k)}\cdot m)= \sum_{\substack{m'\in {\mathcal
%%       T}_{n-k-1}\\w'\in{\mathscr T}_{n-k-1}}} \left( \frac{b-1}{\esp
%%   N}\right )^{n-k-1-\varsigma(m')} \prod_{j=0}^{n-k-2}
%% a^2_{w_{j+1}}(vi\cdot w|_j,0^{(k)}m\cdot m'|_j).
%%   $$
\end{remark}

\subsection{A CLT for random finitely additive measures} 
Suppose again that the Mandelbrot martingale $(U_n)_{n\ge 1}$
converges almost surely and in $L^1$.

Let $\xi$ stand for a centered normal vector with covariance matrix
$A$ and independent of $U$. Then consider a family
$(U(w),\xi(w))_{w\in\bigcup_{n\ge 1}\mathbb N_+^n)}$ of independent
copies of $(U,\xi)$. Also, consider $(a(w))_{w\in\bigcup_{n\ge
    0}\mathbb N_+^n}$ a family of independent copies of $a$, which is
independent of $(U(w),\xi(w))_{w\in\bigcup_{n\ge 1}\mathbb N_+^n}$. By
the calculation of the above paragraph, for all $w\in \bigcup_{n\ge
  1}\mathbb N_+^n$, the sequence of random variables
$$ X_n(w)=\sum_{k=1}^{|w|}\sum_{|v|=k} \left (\prod_{j=0}^{k-1}
a_{v_{j+1}}(w\cdot v_{|j})\right ) \sqrt{U(w\cdot v)}\, \xi(w\cdot
v)\quad (n\ge 1),
$$ which, conditionally on $\sigma (\{a(w),U(w)\})$, is a martingale
bounded in $L^2$ converging almost surely to a random variable $X(w)
$, which is a centered normal vector whose covariance matrix equals
$(b-1)^{-1} \cdot A$ times a copy of $U$. In other words
$X(w)=(b-1)^{-1/2}\sqrt{\widetilde U(w)} \widetilde \xi(w)$, where
$\widetilde U(w)\sim U$, $\widetilde \xi(w)\sim \xi$, and $\widetilde
U(w) $ and $\widetilde \xi(w)$ are independent. Moreover, the random
vectors $( \widetilde U(w),\widetilde \xi(w))_{w\in \mathbb N_+^n}$
are independent, and also independent of
$\sigma(\{a(w_{k-1}),U(w),\xi(w): w\in \bigcup_{k= 1}^n\mathbb
N_+^k)\})$.

\medskip

Also, we have the relation 
$$
X(w)=\sum_{|i|=1} a_{i}(w)\big ( X(wi)+ \sqrt{U(w\cdot i)}\, \xi(w\cdot i)\big ).
$$
Consequently,
$$ M(w)=\left(\prod_{j=0}^{|w|-1} a_{w_{j+1}}(w_{|j})\right) \Big
(\sqrt{\widetilde U(w)} \widetilde \xi(w)+
\sqrt{b-1}\sum_{j=1}^{|w|}\sqrt{U(w_{|j})}\xi( w_{|j})\Big )
$$ defines a $\mathbb R^2$-valued random finitely additive measure on $\mathscr{T}$. We
notice that this measure is obtained as the limit of a mixture of
additive and multiplicative cascades.

\medskip 

Now suppose that $\gamma\in (1,\beta)$ and
$\mu\in\mathscr{P}_\gamma$. For each $n\ge 0$, a complex or
$\mathbb{R}^2$-valued random measure on $\mathscr{T}$ is naturally
associated with $\T^n(\mu)$: consider two independent sequences
$\bigl(W_n(w)\bigr)_{w\in {\mathscr T}}$ and $\bigl(a(w)\bigr)_{w\in
  {\mathscr T}}$ of independent variables equidistributed
with~$W_n\sim\T^n(\mu)$ or~$a$, and for $w\in \mathscr{T}$ define
\begin{equation*}
W_{n+1}(w) =\lim_{p\to\infty} \sum_{|v|=p} \prod_{k=0}^{p-1}
a_{v_{k+1}}(w\cdot v|_{k})W_n(w\cdot v|_{k+1}). 
\end{equation*}
Then define 
$$
\nu_{n}(w)=  W_{n+1}(w)\prod_{k=0}^{|w|-1}
a_{j_{k+1}}(w|_{k})W_n(w|_{k+1}).
$$ When $\mu$ is supported on $\mathbb R_+$ this measure coincides
with the restriction to cylinders of so-called Mandelbrot measure
supported on the boundary of $\mathscr{T}$ and associated with the
family of vectors $(a_{i}(w), W_n(w\cdot i))_{i\ge 1}$, $w\in
\mathscr{T}$.

Also, let $\nu$ be the conservative Mandelbrot measure built from the
family of vectors $(a_{i}(w))_{i\ge 1}$, $w\in \mathscr T$, i.e.,
$\nu=\nu_0$ when $\mu=\delta_{1}$.

It is then almost direct to get the following result from Theorems~\ref{thm13} and~\ref{thm16}. :
\begin{theorem}\label{CLT for measures}
Suppose that either the assumptions of Theorem~\ref{thm16} are
fulfilled or those of Theorem~\ref{thm13} are fulfilled if $W_0$ is real valued. Conditionally on $\mathscr{T}$, for each $p\ge 1$,
$\displaystyle \left((b-1)^{(n+1)/2}
(\nu_n(w)-\nu(w))\right)_{w\in\mathscr{T}_p}$ converges in law to
$(M(w))_{w\in\mathscr{T}_p}$ as $n\to\infty$, with
$A=\begin{pmatrix}x&z\\z&y\end{pmatrix}$.
\end{theorem} 
Indeed, given $p\ge 1$, for $n\ge p$, and $w\in\mathscr{T}_p$, we can write 
$$
\nu_n(w)-\nu(w)= \left (\prod_{k=0}^{p-1}
a_{j_{k+1}}\right) \left (W_{n+1}(w)\prod_{k=1}^{p} W_n(w|_{k}) -1\right ).
$$
Moreover, 
$$
W_{n+1}(w)\prod_{k=1}^{p} W_n(w|_{k})
-1=a_{n+1}(W_{n+1}(w)-1)+\sum_{k=1}^{p} a_{n,k} (W_n(w|_{k}) -1), 
$$ where the random variables $a_{n+1}$ and $a_{n,k}$ are products of
$p$ independent random variables all converging to 1 in law as
$n\to\infty$ and uniformly bounded in $L^2$. Also, due to
Theorem~\ref{thm13}, for each $1\le k \le p$,
$(b-1)^{n/2}(W_{n}(w_{|k})-1)$ converges in law to a copy
$\sqrt{U(w_{|k})}\xi(w_{|k})$ of $\sqrt{U}\xi$, as well as
$(b-1)^{(n+1)/2}(W_{n+1}(w)-1)$ to such a $\sqrt{\widetilde
  U(w)}\widetilde \xi(w)$. Due to the independence properties of the
random variables defining $\nu_n$, we get the desired conclusion.

\begin{remark}
It seems not clear at the moment to associate a functional CLT with a
general distribution for $N$.
\end{remark}

\subsection{A Functional CLT for random functions on [0,1]}
 We still assume that we are in the quadratic case. Moreover, we
 assume that $\mathscr{T}$ is a $c$-adic tree, with $c\ge 2$ (this
 means that $a_j=0$ for $j>c$); it is easily seen that one must have
 $c\ge b$.  If $w\in \mathscr{T}$, the closed $c$-adic interval
 naturally encoded by $w$ is denoted $I_w$, and given a complex or
 $\mathbb R^2$-valued function $f$ defined over $[0,1]$, the increment
 of $f$ over $I_w$ is denoted $\Delta (f,I_w)$.

\medskip

Suppose that $q=1$ and the martingale $(U_n)_{n\ge 1}$ converges in
$L^{1+\epsilon}$. Then, with the notations of the previous section, it
is direct that for all $p\ge 1$,
$$ \esp \sum_{w\in\mathscr{T}_p} |M(w)|^2= \mathrm{O}\Big (p^2\Big (
\sum_{i=1}^ca_i^2\Big)^p\Big )= \mathrm{O}(c^{-p\gamma})
$$ for some $\gamma>0$. It follows that $\sup_{w\in \mathscr{T}_p}
|M(w)|$ tends to 0 exponentially fast. Consequently, the process $F$
defined on the $c$-adic numbers of $[0,1]$ by $F(0)=0$ and $\Delta
(F,I_w)=M(w)$ (this definition is consistent since $M$ is a measure)
extends to a unique complex-valued H\"older continuous function over
$[0,1]$, still denoted $F$.

\medskip

Now suppose $b>2$, $\gamma\in (1,\beta)$ and
$\mu\in\mathscr{P}_\gamma$. At first, the complex-valued process
defined for each $n\ge 1$, by $G_n(0)=0$ and $\Delta
(G_n,I_w)=\nu_n(w)$ for each $w\in \mathscr{T}$ can be shown to extend
to a unique H\"older continuous function (see \cite[Theorem
  2.1]{BJM}), hence if $\mu\neq\delta_{1}$, the relations $F_n(0)=0$
and $\Delta (F_n,I_w)=\displaystyle (b-1)^{(n+1)/2}(\nu_n(w)-\nu(w))$
define a unique H\"older continuous function.

\begin{theorem}\label{CLTFunc}
Suppose that either the assumptions of Theorem~\ref{thm16} are
fulfilled or those of Theorem~\ref{thm13} are fulfilled if $W_0$ is real valued. Suppose also that $\mathscr{T}$ is a $c$-adic tree. The sequence $(F_n)_{n\ge 1}$
converges in law to $F$ as $n\to\infty$.
\end{theorem}

In \cite{BPW} we obtained this result when $a_j=c^{-1}$ for all $j$,
in which case $U=1$ almost surely, and with $\mu$ supported on
$\mathbb R_+$, which implies that $y=z=0$ and so $F$ is real valued.

\begin{proof} 
Due to Theorem~\ref{CLT for measures} we only have to prove the
tightness of the sequence of laws of the functions $F_n$, $n\ge
1$. This is quite similar to the proof of Proposition~9 in~\cite{BPW},
but for reader's convenience we include some details. The same
arguments as those used to prove Theorem~\ref{CLT for measures} imply
that for all $n\ge 1$, for all $p\ge 1$
$$ \esp\Big (\sum_{w\in\mathscr{T}_p} |\Delta (F_n,I_w)|^2\Big)=
  \mathrm{O}\Big (p^2\Big ( \sum_{i=1}^ca_i^2\Big)^p\Big )=
  \mathrm{O}(p^2c^{-p \gamma}),
$$
where $ \mathrm{O}$ is uniform with respect to $n$. For any $t>0$ this yields 
$$ \mathbb P(\exists \, w\in \mathscr T_p,\ |\Delta (F_n,I_w)|\ge t
\,c^{-p\gamma/4})\le \mathrm{O}(t^{-2} p^2c^{-p \gamma/2}).$$ Let
$\omega(F_n,\cdot)$ stand for the modulus of continuity of $F_n$.  Fix
$\varepsilon>0$. It is standard that if $\delta\in (0,1)$ and
$p_\delta=-\log_c (\delta)$,
\begin{eqnarray*}
\bigl\{\omega( F_n,\delta)\ge 2(c-1)\,\varepsilon \bigr\}
&\subset& \left\{ \sum_{p\ge p_\delta} \sup_{w\in \mathscr T_p}
\Delta(F_n,I_w) > \varepsilon\right\}\\
&\subset& \bigcup_{p\ge p_\delta}\left\{ \sup_{w\in  \mathscr T_p}
\Delta(F_n,I_w) >
(1-c^{-\gamma/4})\,c^{p_\delta \gamma/4}\,\varepsilon\,c^{-p\gamma/4}\right\},
\end{eqnarray*}
so 
$$ \mathbb P(\omega( F_n,\delta)\ge
2(c-1)\,\varepsilon)=\mathrm{O}\left (\frac{c^{-p_\delta
    \gamma/2}}{(1-c^{-\gamma/4})^2\varepsilon^2}\sum_{p\ge p_\delta}
p^2c^{-p \gamma/2}\right )
$$
uniformly in $n\ge 1$. Consequently, 
\begin{equation*}
\lim_{\delta\to 0} \sup_{n\ge 1} \mathbb{P}\bigl(\omega
(F_n,\delta)>2(c-1)\,\varepsilon\bigr) =0,
\end{equation*}
which yields the desired tightness (see \cite{Bil}).

\end{proof}

\subsection{Multifractal analysis of the increments of the limit
  process $F$} We work under the assumptions of the previous section
defining $F$ as a non trivial H\"older continuous function.  At first
we notice that
$$ \Delta(F,I_w)=\nu(w)\, \Big (\sqrt{\widetilde U(w)} \widetilde
\xi(w)+ \sqrt{b-1}\sum_{j=1}^n\sqrt{U(w_{|j})}\xi( w_{|j})\Big ).
$$ To simplify the purpose, we assume that the $a_i$, $i\ge 1$ do not
vanish almost surely. If $x\in (0,1)$, denote by $w_n(x)$ the $c$-adic
word $w$ of generation $n$ encoding the unique semi-open to the right
$c$-dic interval which contains $x$. The sequence $\displaystyle\Big
(\frac{\log(\nu(w_n(x)))}{n},\frac{\Delta(F,I_{w_n(x)}))}{\sqrt{b-1}n
  \, \nu(w_n(x))}\Big )$, $n\ge 1$, provides a fine description of the
asymptotic behavior of $\Delta(F,I_{w_n(x)})$. It is essentially the
$\mathbb R^3$-valued branching random walk with independent components
associated with the random vectors
$(\log(a_i(w)),\sqrt{U(wi)}\xi(wi))_{1\le i\le c}$, $w\in\mathscr{T}$.

\medskip

For all subset $K$ of $\mathbb{R}^2$ set 
$$ E(K)= \left\{x\in (0,1): \bigcap_{N\ge 1}\overline{\Big \{\Big
  (\frac{\log(\nu(w_n(x)))}{n},\frac{\Delta(F,I_{w_n(x)}))}{\sqrt{b-1}n\,
    \nu(w_n(x))}\Big ):n\ge N\Big\}}=K\right\}.
$$

For $(q_1,{q}_2)\in\mathbb R\times \mathbb R^2$, set 
$$
P(q_1,{q}_2)=\log\Big (\esp \sum_{i=1}^c a_i^{q_1}\Big ) + \log \esp
e^ {\langle {q}_2 |\sqrt{U}\xi\rangle},%=\log\Big (\esp \sum_{i=1}^c
                                       %a_i^{q_1}\Big ) + \log \Big
                                       %(c\esp e^ {-q^2_2 {U}/2}\Big
                                       %),
$$
and for $(\gamma_1,\mathbf{\gamma}_2)\in\mathbb R\times \mathbb R^2$, let 
$$ P^*(\gamma_1,\gamma_2)=\inf\{P(q_1,q_2)-\gamma_1q_1-\langle
\gamma_2|q_2\rangle: (q_1,q_2)\in\mathbb R\times \mathbb R^2\}
$$ 
be the concave Legendre transform of $P$ at $(\gamma_1,\gamma_2)$. 

%Notice that the condition \begin{equation}\label{momentsUxi}
%\displaystyle (b-1)\, \mathrm{ess\,sup} \sup_{1\le j\le c}
%a_j^2<1 \end{equation} is sufficient to have that all the moments of
%positive orders of $U$, hence of $\sqrt{U}|\xi|$, are finite and
%control the term $\sqrt{\widetilde U(w)} \widetilde \xi(w)$ in
%$\Delta(F,I_w)$.

\medskip

As a consequence of the general study of the multifractal behavior of
vector valued branching random walks achieved in \cite{AB}, we have:
\begin{theorem}
With probability 1, for all compact connected subsets $K$ of
$\mathbb{R}^3$, we have
$$
\dim E(K)= \frac{1}{\log(c)} \inf\{P^*(\gamma): \gamma\in K\},
$$ where $\dim$ stands for the Hausdorff dimension and a negative
dimension means that $E(K)$ is empty.  \end{theorem} This is a non
trivial extension of the result obtained in \cite{BPW} in the case
where $a_j=c^{-1}=b^{-1}$ for all $j$ (hence $\nu$ is the Lebesgue
measure and $U=1$) and $\mu$ is supported on $\mathbb R_+$, which
implies that the multifractal analysis reduces back to that of a
centered Gaussian branching random walk in $\mathbb R$.

\begin{remark}
It is worth specifying that $\esp e^ {\langle {q}_2
  |\sqrt{U}\xi\rangle}|U=e^{UQ(q_2)}$, where $Q$ is a non-negative
non-degenerate quadratic form which is positive definite if and only
if $A$ is invertible. Moreover, we saw in Section~\ref{clt61} that $A$
is  invertible  if and only if $\mathbb P(W\in\mathbb
C\setminus\mathbb R)>0$.
In addition,  by \cite[Theorem 2.1]{Liu96} one
has three situations for the behavior of the moment generating
function of $U$:

(1) there exists $p>1$ such that $\esp U^p=\infty$,
i.e. $((b-1)^p+1)\esp \sum_{i\ge 1}a_j^{2p}\ge 1$. Then, for $r\ge 0$
one has $\esp e^{rU}<\infty$ only if $r=0$. In particular, if $Q$ is
positive definite, then the domain of $P$ reduces to $\mathbb
{R}\times\{(0,0)\}$.

(2) If $\esp U^p<\infty$ for all $p\ge 1$ and
$\mathrm{ess\,sup}((b-1)^{p_0}+1)\sum_{i= 1}^ca_j^{2p_0}<1$ for some
$p_0>1$, then $\esp e^{rU}<\infty$ for all $r\ge 0$, and the domain of
$P$ is $\mathbb R^3$.

(3) If $\esp U^p<\infty$ and $\mathrm{ess\,sup}((b-1)^{p}+1)
\sum_{i\ge 1}a_j^{2p}\ge 1$ for all $p\ge 1$, then $\esp
e^{rU}<\infty$ for some $r>0$, hence the domain of $P$ contains
$\mathbb{R}\times V$, where $V$ is a neighborhood of $(0,0)$.

\end{remark}

\section{Final remarks on fixed points of nonlinear smoothing
  transformations} 

One can wonder whether $\smooth_2$ may have fixed points in the space
$\mathcal P_+$ of probability measures on $\mathbb R_+$ with infinite
first moment, like $\smooth_1$ does (see
\cite{DuLi,liu98,BiKy3,AlBiMe}. It turns out that this is not the case
under mild conditions. Suppose that $\#\{i\ge 1: a_i>0\}$ is
bounded. Let $\mu\in \mathcal P_+$ be a fixed point of $\smooth_2$
with $ \moment_1(\mu)=\infty$. Write the equality $\mu=\smooth_2
(\mu)$ in the form $Y=\sum_{i\ge 1} (a_i Y_i)\widetilde Y_i$, so that
$\mu$ is a fixed point of the mapping $\smooth_1$ defined with
$(\widetilde a_i=a_i Y_i)_{i\ge 1}$. We can use the theory of the
fixed point of $\smooth_1$ \cite{DuLi,AlMe12} to claim that there must
exist a unique $\alpha\in (0,1]$ such that $\esp \sum_{i\ge 1}
a_i^\alpha Y_i^\alpha=1$ and $\esp \sum_{i\ge 1} a_i^\alpha Y_i^\alpha
\log (a_i^\alpha Y_i^\alpha)\le 0$. In particular, $\esp
Y^\alpha<\infty$, so $\alpha<1$. Moreover, there exists a random
variable $Z$, namely a fixed points of the mapping $\smooth_1$ defined
with $(a_i^\alpha Y_i^\alpha)_{i\ge 1}$, such that
$Y=\mathscr{L}(Z^{1/\alpha} X)$, where $X$ is a positive stable law of
index $\alpha$. In particular, $\esp Y^\alpha=(\esp Z)\esp
X^\alpha=\infty$, which is a contradiction.\medskip

Let us close the paper with a natural question: under what necessary
and sufficient condition on $(a_i)_{i\ge 1}$ does the  martingale
$(Z_n)_{n\ge 1}$ of Section~\ref{existence} converge to a non
degenerate random variable~$Z_\infty$?  

At least we know that $\esp Z_\infty=(\esp Z_\infty)^2$, hence in case
of non degeneracy the convergence holds in $L^1$.

\end{document}